\documentclass[a4paper,english,11pt]{elsarticle}
\usepackage[T1]{fontenc}
\usepackage[utf8]{inputenc}
\usepackage[english]{babel}
\usepackage{amsmath,amsfonts,amssymb,amsthm,mathrsfs}
\usepackage{latexsym}
\usepackage{enumerate}  
\usepackage{graphicx}
\usepackage{epstopdf}
\usepackage{float}
\usepackage{booktabs}
\usepackage{hyperref}
\usepackage[margin=10pt,font=small]{caption}
\usepackage{xcolor}
\usepackage{listing}
\usepackage{matlab-prettifier}
\newcommand{\R}{\mathbb{R}}
\newcommand{\Z}{\mathbb{Z}}

\newtheorem{theorem}{Theorem}[section]

\newtheorem{lemma}{Lemma}[section]

\theoremstyle{definition}

\theoremstyle{remark}
\newtheorem*{remark}{Remark}
\providecommand{\keywords}[1]
{
  \small	
  \textbf{\textit{Keywords---}} #1
}
\begin{document}
\title{A pseudo-spectral splitting method for linear dispersive problems with transparent boundary conditions}
\author{Lukas Einkemmer\fnref{fn1}}
\ead{lukas.einkemmer@uibk.ac.at}
\author{Alexander Ostermann\fnref{fn1}}
\ead{alexander.ostermann@uibk.ac.at}
\author{Mirko Residori\fnref{fn1}}
\ead{mirko.residori@uibk.ac.at}
\fntext[fn1]{Department of Mathematics, University of Innsbruck \protect}
\begin{abstract}
The goal of the present work is to solve a linear dispersive equation
with variable coefficient advection on an unbounded domain. In this
setting, transparent boundary conditions are vital to allow
waves to leave (or even re-enter) the, necessarily finite,
computational domain. To obtain an efficient numerical scheme we
discretize space using a spectral method. This allows us to
drastically reduce the number of grid points required for a given
accuracy. Applying a fully implicit time integrator, however, would
require us to invert full matrices. This is addressed by performing an
operator splitting scheme and only treating the third order
differential operator, stemming from the dispersive part, implicitly;
this approach can also be interpreted as an implicit-explicit
scheme. However, the fact that the transparent boundary conditions are
non-homogeneous and depend implicitly on the numerical solution
presents a significant obstacle for the splitting/pseudo-spectral
approach investigated here. We show how to overcome these difficulties
and demonstrate the proposed numerical scheme by performing a number
of numerical simulations.
\end{abstract}
\maketitle

\keywords{splitting methods, linear dispersive problems, pseudo-spectral methods, transparent boundary conditions}

\section{Introduction}
In this paper, we consider a linear dispersive problem with a variable coefficient advection in one space dimension 
\begin{equation}
\label{eqo}
u_t + g(x)u_x + u_{xxx} = h(t,x),\quad t\in [0,T],\,x\in\R.
\end{equation}
This partial differential equation (PDE) consists of an advection part given by $g(x)u_x$ and a dispersive part given by $u_{xxx}$. Applications can be found, for example,~in modelling long waves in shallow water with an uneven bottom (see, for example, \cite{kakutani71,whitham74}). 

We face several difficulties in designing an efficient numerical method to solve this problem. The third order differential operator causes explicit methods to take excessively small time steps. This could be remedied by employing an implicit method, such as the Crank--Nicolson scheme that is used, for example,  in \cite{besse16}. In this paper, our goal is to use a spectral approach to discretize space. This drastically reduces the number of grid points needed. However, the variable coefficient advection results in a full matrix. Thus, applying an implicit method to the entire problem is very costly in terms of computational effort. We, therefore, choose to employ a splitting approach to separate the dispersive part, which will be treated implicitly, from the advection part, which will be treated explicitly. Operator splitting schemes have been employed for dispersive problems before and we refer the reader to the literature \cite{einkemmer15,einkemmer18,holden11,klein11,rouhi95}. Finally, problem \eqref{eqo} is posed on an unbounded domain. However, to perform numerical simulations the domain has to be restricted to a compact subset of the real line. Our goal is to derive a numerical scheme that, despite this truncation, retains the dynamics of the original problem. This is accomplished using so-called transparent boundary conditions (TBCs). Transparent boundary conditions are derived in such a manner that a wave propagating to either the left or the right can leave the computational domain without producing reflections. 
A previous work which combines splitting scheme with absorbing boundary conditions for the Schr\"odinger equation can be found in \cite{bertoli17}.
We further note that combining TBCs with a pseudo-spectral approach is not entirely straightforward as TBCs can not be formulated as Dirichlet or Neumann boundary conditions and the value imposed depends on the history of the numerical solution at the boundary.

This work is inspired by the paper of Besse et al.~\cite{besse16}. In their work the $\mathcal{Z}\text{-transform}$ is used to compute transparent boundary conditions for a third-order linear problem with constant coefficients. The temporal discretization is carried out by the Crank--Nicolson scheme and the spatial discretization by finite differences. Differently from their approach we will follow a splitting strategy in order to  divide the full equation into its dispersive part $u_t + u_{xxx}=0$ and its advection part $u_t+g(x)u_x=0$. The transparent boundary conditions are then derived in a semi-discrete setting using the backward Euler and the explicit Euler method for the dispersive and the advection equation, respectively. Let us also remark that, since we limit ourselves to the first order Lie--Trotter splitting in this work, the resulting numerical time integrator can also be written as an implicit-explicit (IMEX) scheme. The space discretization is performed by a pseudo-spectral approach similarly to what is described in \cite{shen04,zheng08}. This results in super-polynomial convergence in space. Consequently very accurate numerical solutions are achieved by using only a small number of grid points. 

Due to the non-locality in time (in 1-D) and space (2-D or higher dimensions), transparent boundary conditions are expensive to compute. While it is possible to employ them in 1-D, the 2-D case becomes impracticable and one needs an approximation of these conditions, the so called \emph{absorbing} boundary conditions (ABCs). In recent years a lot of work has been done for the Schr\"odinger equation coupled with transparent boundary conditions, see \cite{antoine08, arnold03}. For third order problems less literature is available. We refer the reader to \cite{besse16,besse16a,zheng08}. 

The paper is organized as follows. In Section 2 we derive a semi-discrete scheme, discrete in time and continuous in space, by applying the Lie--Trotter splitting. In Section 3 we impose transparent boundary conditions for the scheme derived in Section 2. In particular, we determine the proper values of the numerical solution at the boundaries with the help of the $\mathcal{Z}$-transform. In Section 4 we describe a pseudo-spectral method for the spatial discretization which takes the TBCs into account. Finally, in Section 5 we present some numerical results that illustrate the theoretical findings. 

\section{Time discretization: a splitting approach}

In this section we derive a semi-discrete scheme by applying Lie--Trotter splitting to~\eqref{eqo}. Splitting methods give us the possibility to choose different numerical methods in order to solve the advection part and the dispersive part.  
For the dispersive part it is convenient to employ an implicit time integrator. Indeed, the third derivative in space makes the Courant--Friedrichs--Lewy (CFL) condition very strict for any explicit scheme. In this work we choose the implicit Euler scheme for the dispersive part. For the advection part we use a different time integration scheme. A possible choice is the explicit Euler method, since the CFL condition given by the advection term is generally not prohibitive. 
Both the implicit and the explicit Euler method converges with order one. Numerical methods of order one are a sufficient choice in our case because of the bottleneck given by the Lie--Trotter splitting, which converges with order one only. 

We consider problem \eqref{eqo} supplemented with initial data $u_0$ 
\begin{equation}
\label{eq0}
\begin{cases}
u_t + g(x) u_x + u_{xxx} = h(t,x),\quad (t,x)\in[0,T]\times\R,\\
u(0,x) = u_0(x).\\
\end{cases}
\end{equation}
The initial data $u_0$ and the source term $h$ are assumed to be compactly supported.  The coefficient $g(x)$ is constant outside a bounded interval. For simplicity of exposition we assume $h=0$. Let $u^m(x)$ be a numerical approximation to the exact solution $u(t,x)$ of \eqref{eq0} at time $t=t_m=m\tau$. We split problem \eqref{eq0} into two sub-problems:
\vspace{2mm}

\begin{minipage}{0.45\textwidth}
\begin{equation}
\label{eq1}
\begin{cases}
v_t + g(x) v_x = 0,\\
v(0,x) = v_0(x),\\
\end{cases}
\end{equation}
\end{minipage}
\begin{minipage}{0.45\textwidth}
\begin{equation}
\label{eq2}
\begin{cases}
w_t + w_{xxx} = 0,\\
w(0,x) = w_0(x).\\
\end{cases}
\end{equation}
\end{minipage}

\vspace{2mm}
Let us denote by  $\Psi^{\tau}(v_0)$ and $\Phi^{\tau}(w_0)$ the flows of \eqref{eq1} and \eqref{eq2}, respectively. We approximate the solution $u(t,x)$ at time $t=t_{m+1}$ starting from $u^m(x)~\approx~u(t_m,x)$ by applying a Lie--Trotter splitting, i.e.
\[
u^{m+1}(x) = \Phi^{\tau} \circ \Psi^{\tau} \big(u^{m}\big)(x).
\]
The semi-discrete numerical schemes for \eqref{eq1} and \eqref{eq2}, respectively, take the form
\begin{equation}
v^{m+1/2} + \tau g\, u^{m}_x = u^m,\quad u^{m+1} + \tau u^{m+1}_{xxx} = v^{m+1/2},\quad m\geq 0.\\
\end{equation}
Composing the two flows by Lie--Trotter splitting we obtain the numerical scheme
\begin{equation}
\label{eq4}
u^{m+1} + \tau u^{m+1}_{xxx} = v^{m+1/2} = u^{m} - \tau g(x)\, u_x^{m},\quad u^0(x) = u(0,x).
\end{equation}

\begin{remark}
In this work the order of composition of the flows $\Psi^{\tau}$ and $\Phi^{\tau}$ is important. When we exchange the flow's composition we obtain 
\[
u^{m+1}(x) = \Psi^{\tau} \circ \Phi^{\tau} \big(u^{m}\big)(x).
\]
The numerical solution for $\Phi^{\tau}\big(u^m\big)$ is given by implicit Euler, we have
\begin{equation}
\label{eqmm1}
v^{m+1/2} + \tau v^{m+1/2}_{xxx} = u^{m}.
\end{equation}
The numerical solution for $\Psi^{\tau}\big(v^{m+1/2}\big)$ is given by explicit Euler, we have
\begin{equation}
\label{eqm1}
u^{m+1} = v^{m+1/2}-\tau g v_x^{m+1/2}.
\end{equation}
It is not possible from \eqref{eqmm1}-\eqref{eqm1} to obtain a simple formula which relates $u^{m+1}$ to $u^m$ as in \eqref{eq4}. Therefore, the numerical scheme in \eqref{eq4} is easier to handle.
\end{remark}
\section{Transparent boundary conditions}\label{tbc}
Transparent boundary conditions are non-local in time. Therefore, they cannot be computed separately for the two sub-problems of the Lie--Trotter splitting. The idea is first to derive the numerical scheme \eqref{eq4} and then to compute the boundary conditions for that particular scheme. The drawback of this approach is that different time discretizations require different boundary conditions. On the other hand, the computed boundary conditions produce (theoretically) no reflections since they are perfectly designed for the employed numerical scheme. 

The mathematical tool we use to derive the transparent boundary conditions is the $\mathcal{Z}$-transform. The $\mathcal{Z}$-transform requires an \emph{equidistant} time discretization. 
Further, we assume that 
\begin{equation}
\label{eqm2}
u^m(x)\to 0,\, |x|\to\infty,\quad \text{for every } m\geq0.
\end{equation}
We refer to \eqref{eqm2} as the \emph{decay condition}. 
Given a sequence $\mathbf{u}=\{u^l\}_{l\geq0}$ its $\mathcal{Z}$--transform is defined by
\[
\hat{u}(z):=\mathcal{Z}(\mathbf{u})(z) = \sum_{l=0}^{\infty} z^{-l} u^l,\quad z\in\mathbb{C},\, |z|>\rho\geq 1,
\]
where $\rho$ is the radius of convergence of the series.
We recall the main properties of the $\mathcal{Z}$-transform that are used in the following. 
\begin{itemize}
\item \emph{Linearity}: for $\alpha,\beta\in\mathbb{R}$, $\mathcal{Z}(\alpha\mathbf{u}+\beta\mathbf{v})(z) = \alpha\hat{u}(z)+ \beta\hat{v}(z)$;
\item \emph{Time advance}: for $k>0$, $\mathcal{Z}(\{u^{l+k}\}_{l\geq0})(z) = z^k\hat{u}(z)-z^k\sum_{l=0}^{k-1}z^{-l}u^l$;
\item \emph{Convolution}: $\mathcal{Z}\big(\mathbf{u} *_d \mathbf{v}\big)(z) = \hat{u}(z)\hat{v}(z)$;
\end{itemize}
where $*_d$ denotes the discrete convolution
\[
(\mathbf{u} *_d \mathbf{v})^m:= \sum_{j=0}^m u^jv^{m-j},\quad m\geq 0.
\] For more details and properties about the $\mathcal{Z}$-transform we refer the reader to \cite{besse16}.

Now we determine the transparent boundary conditions for the scheme \eqref{eq4}. For this we assume that the coefficient $g(x)$ is constant for $x\leq\alpha$ and $x\geq\beta$ and that $u^0$ is continuous and its support is contained in $[\alpha,\beta]$. Without loss of generality, we further assume $\alpha=-1$ and $\beta=1$. We then split the domain into three parts 
\[
\mathbb{R}=(-\infty,-1]\cup(-1,1)\cup [1,\infty).
\]
Let $\Omega_-=(-\infty,-1]$ and $\Omega_+=~[1,\infty)$ be the two outer domains. 
Note that $u^{0}(x)~=0$ for $x\in \Omega_-\cup\Omega_+$. We now consider the problem in the two outer domains. We denote by $g_{\pm}$ the constant values that $g$ assumes for $x\in \Omega_-$ and $x\in \Omega_+$, respectively. 
Computing the $\mathcal{Z}$-transform of the sequence $\{u^m(x)\}_{m\geq 0}$, where $x\in\Omega_-\cup\Omega_+$ plays the role of a parameter, we get 
\[
\mathcal{Z}(\{u^{m+1}(x)\}_{m\geq0})(z)= z\hat{u}(x,z)-zu^0(x) = z\hat{u}(x,z),\quad x\in\Omega_-\cup\Omega_+,
\]
where we used the time advance property of the $\mathcal{Z}$-transform and the assumption of a supported $u^0(x)$ in $\Omega$. 
By applying the $\mathcal{Z}$-transform to \eqref{eq4}, we obtain an ordinary differential equation in the variable $x$:
\begin{equation}
\label{eq5}
z (\hat{u} + \tau\hat{u}_{xxx}) = \hat{u} - \tau g_{\pm}\,\hat{u}_x,\quad x\in \Omega_-\cup\Omega_+. 
\end{equation}
The solution of this equation for $x\in \Omega$ is obtained by employing an exponential ansatz. It is given by
\begin{equation}
\label{eq6}
\hat{u}(x,z) = \sum_{j=1}^3 c_j(z) e^{r_j(z)x},
\end{equation}
where $c_j$ are coefficients and $r_j$, $j=1,2,3$ are the roots of the characteristic polynomial associated to \eqref{eq5}, i.e. 
\begin{equation}
\label{cpol}
z\tau r^3 + \tau g_{\pm}\, r + z-1 = 0.  
\end{equation}
The roots $r_j$ can be computed analytically. They are given in \cite{besse16} by
\[
\begin{split}
& r_j(z) = \omega^{j-1}\,\zeta(z)-\frac{g_{\pm}}{3z\omega^{j-1}\,\zeta(z)},\quad \omega = e^{i2/3\pi},\quad j=1,2,3,\\
& \zeta (z) = -\bigg(\frac{G(z)}{2}\bigg)^{1/3},\\
& G(z) = \frac{z-1}{z\tau} + \sqrt{\bigg(\frac{z-1}{z\tau}\bigg)^2+\frac{4}{27}\bigg(\frac{g_{\pm}}{z}\bigg)^3}.
\end{split}
\]
In the same paper \cite{besse16} the following property is shown:
\begin{theorem}
The roots $r_j(z)$ for $j\in\{1,2,3\}$ are such that $\mathrm{Re}\,r_1(z)<0$ and $\mathrm{Re}\,r_{2,3}(z)>1$, $z=\rho \mathrm{e}^{i\theta}$, $\theta\in[0,2\pi)$ and for a fixed $\rho>1$. 
\end{theorem}
The solution \eqref{eq6} must be coupled with the \emph{decay condition} $u^m(x)\to 0,\, |x|\to~\infty$. Therefore,
we choose the coefficients $c_j(z)$ such that the decay condition is satisfied. This requires $c_1(z) = 0$ for $x\in\Omega_-$. Otherwise the solution would grow to infinity as $\mathrm{Re}\,r_1(z)<0$. Similarly for $x\in\Omega_+$ we impose $c_{2}(z)=c_{3}(z) = 0$. Otherwise the solution would grow to infinity as $\mathrm{Re}\,r_{2,3}(z)>0$.  Finally we obtain 
\begin{equation}
\label{eq7}
\begin{split}
& \hat{u}(x,z) = c_1(z) e^{r_1(z)x},\quad x\in\Omega_+,\\
& \hat{u}(x,z) = c_2(z) e^{r_2(z)x} + c_3(z) e^{r_3(z)x},\quad x\in\Omega_-.\\
\end{split}
\end{equation}

The coefficients $c_j$ are not yet determined. Notice that differentiating the equation \eqref{eq7} w.r.t. $x$ we obtain
\begin{equation}
\label{eq8}
\begin{split}
& \hat{u}_x(x,z) = c_1(z) r_1(z)\, e^{r_1(z) x} = r_1(z)\hat{u}(x,z),\quad x\in\Omega_+,\\
& \hat{u}_x(x,z) = c_2(z) r_2(z)\,e^{r_2(z) x} + c_3(z) r_3(z) e^{r_3(z) x},\quad x\in\Omega_-
\end{split}
\end{equation}
and differentiating once more we get
\begin{equation}
\label{eq9}
\begin{split}
&\hat{u}_{xx}(x,z) = r_1^2(z)\, \hat{u}(x,z),\quad x\in\Omega_+,\\
&\hat{u}_{xx}(x,z) = c_2(z) r_2^2(z)\, e^{r_2(z) x} + c_3(z)r_3^2(z)\, e^{r_3(z) x},\quad x\in\Omega_-.
\end{split}
\end{equation}
Rewriting formulae \eqref{eq7}--\eqref{eq9}, we can express the function $\hat{u}$ of \eqref{eq5} and $\hat{u}_x$ for $x\in\Omega_+$  in terms of its second derivative 
\[
\hat{u}(x,z) = \frac{1}{r_1^2(z)}\,\hat{u}_{xx}(x,z),\quad \hat{u}_{x}(x,z) = \frac{1}{r_1(z)}\,\hat{u}_{xx}(x,z),\quad x\in\Omega_+.
\]
The function $\hat{u}$  for $x\in\Omega_-$ can be expressed in terms of its first and second derivative
\[
\hat{u}(x,z) = \bigg(\frac{1}{r_2(z)}+\frac{1}{r_3(z)}\bigg)\,\hat{u}_x(x,z) - \frac{1}{r_2(z)r_3(z)}\,\hat{u}_{xx}(x,z),\quad x\in\Omega_-.
\]
The transparent boundary conditions in the $\mathcal{Z}$-transformed space are simply computed by evaluating $\hat{u}$ at $x=\pm 1$ and $\hat{u}_x$ at $x=1$, namely
\begin{equation}
\label{eq10}
\begin{split}
 \hat{u}(-1,z) &= \bigg(\frac{1}{r_2(z)}+\frac{1}{r_3(z)}\bigg)\hat{u}_x(-1,z) - \frac{1}{r_2(z)r_3(z)}\hat{u}_{xx}(-1,z),\\
 \hat{u}(1,z) &= \frac{1}{r_1^2(z)}\hat{u}_{xx}(1,z),\\
 \hat{u}_x(1,z) &= \frac{1}{r_1(z)}\hat{u}_{xx}(1,z).
\end{split}
\end{equation}
The solution in the physical space is then obtained by applying the inverse $\mathcal{Z}\text{-transform}$ to \eqref{eq10}. 
From \eqref{eq10} we get
\begin{equation}
\label{eqt}
\begin{split}
 \mathbf{u}(-1) &= \mathcal{Z}^{-1}\bigg\{z\mapsto\bigg(\frac{1}{r_2(z)}+\frac{1}{r_3(z)}\bigg)\hat{u}_x(-1,z)\bigg\} - \mathcal{Z}^{-1}\bigg\{z\mapsto\frac{1}{r_2(z)r_3(z)}\hat{u}_{xx}(-1,z)\bigg\},\\
 \mathbf{u}(1) &= \mathcal{Z}^{-1}\bigg\{z\mapsto\frac{1}{r_1^2(z)}\hat{u}_{xx}(1,z)\bigg\},\\
 \mathbf{u}_x(1) &= \mathcal{Z}^{-1}\bigg\{z\mapsto\frac{1}{r_1(z)}\hat{u}_{xx}(1,z)\bigg\}.
\end{split}
\end{equation}
By the convolution property of the $\mathcal{Z}$-transform, we have $\mathbf{u}*_d\mathbf{v} = \mathcal{Z}^{-1}\{z\mapsto\hat{u}(z)\hat{v}(z)\}$. Let us denote by
 \[
\begin{split}
\mathbf{Y}_1 &= \mathcal{Z}^{-1}\bigg\{z\mapsto\frac{1}{r_2(z)}+\frac{1}{r_3(z)}\bigg\},\\
\mathbf{Y}_2 &= \mathcal{Z}^{-1}\bigg\{z\mapsto-\frac{1}{r_2(z)r_3(z)}\bigg\},\\
\mathbf{Y}_3 &= \mathcal{Z}^{-1}\bigg\{z\mapsto\frac{1}{r_1^2(z)}\bigg\},\quad \mathbf{Y}_4 = \mathcal{Z}^{-1}\bigg\{z\mapsto\frac{1}{r_1(z)}\bigg\}.
\end{split}
\]
Equations \eqref{eqt} take the form
\[
\begin{split}
\mathbf{u}(-1) &= \mathbf{Y}_1 *_d \mathbf{u}_{x}(-1) + \mathbf{Y}_2 *_d \mathbf{u}_{xx}(-1),\\
\mathbf{u}(1) &= \mathbf{Y}_3 *_d \mathbf{u}_{xx}(1),\\
\mathbf{u}_x(1) &= \mathbf{Y}_4 *_d \mathbf{u}_{xx}(1).
\end{split}
\]
We rewrite the boundary conditions in order to highlight the part that depends on $m$ (collected on the left hand-side) and the \emph{history} that depends on previous time steps (collected on the right-hand side)
\begin{equation}
\label{eq15}
\begin{split}
u^m(-1) - Y_1^0\,u^m_{x}(-1) - Y_2^0u^m_{xx}(-1) & = \sum_{j=1}^m \big(Y_1^j\, u^{m-j}_{x}(-1) + Y_2^j\, u^{m-j}_{xx}(-1)\big) =: h_1^m,\\
u^m(1) - Y_3^0\,u_{xx}^m(1) &= \sum_{j=1}^m Y_3^j\,u^{m-j}_{xx}(1)=:h_2^m,\\
u_x^m(1) - Y_4^0\,u_{xx}^m(1) &= \sum_{j=1}^m Y_4^j\, u_{xx}^{m-j}(1)=:h_3^m.
\end{split}
\end{equation}
The quantities $\mathbf{Y}_k$, $k=1,\dots, 4$ can be computed numerically as we briefly describe in Section \ref{sec4}. For more details about the $\mathcal{Z}$-transform and its numerical implementation as well as the inverse $\mathcal{Z}$-transform, we refer the reader to \cite{besse16,zisowsky03}.

Summing up, we obtain the following boundary value problem to solve:
\begin{equation}
\label{eqibvp}
\begin{cases}
& u^{m} + \tau u^{m}_{xxx} = u^{m-1} - \tau g(x)\, u_x^{m-1}, \quad x\in(-1,1),\\
& u^{m}(-1) - Y_1^0\,u^{m}_{x}(-1) - Y_2^0u^{m}_{xx}(-1) = h_1^{m},\\
& u^{m}(1) - Y_3^0\,u_{xx}^{m}(1) = h_2^{m},\\
& u_x^{m}(1) - Y_4^0\,u_{xx}^{m}(1) = h_3^{m},\\
& u^0 = u(0,x).
\end{cases}
\end{equation}
The task of solving this problem will be carried out in the next sections.
\section{Pseudo-spectral space discretization}
\label{secPseudo}
The space discretization of problem \eqref{eqibvp} is carried out by a pseudo-spectral method. We implement a dual-Petrov--Galerkin method. This method has been employed for the pure dispersive equation in \cite{zheng08}. 
We remark that other spatial discretizations are also possible. For example, in \cite{besse16} a finite difference method is employed to discretize in space. Finite differences are easy to implement, but an accurate numerical solution is achieved only when a fairly large amount of grid points are used. On the other hand, pseudo-spectral methods require a modest number of points and provide very accurate solutions. These methods have to be carefully designed in order to get banded spatial discretization matrices so that the resulting linear system is cheap to solve.
 
Let $a<b$ and let $\mathcal{P}_N$ be the space of polynomials of degree less or equal than $N$ on the interval $[a,b]$. We define
\begin{equation}
\label{eq19}
\begin{split}
V_N = \{u\in\mathcal{P}_N\, |\,  u(a) - Y_1^0 u_x(a) - Y_2^0 u_{xx}(a) & = 0,\\
u(b) - Y_3^0 u_{xx}(b) & = 0,\\
u_x(b)-Y_4^0 u_{xx}(b) & = 0\},
\end{split}
\end{equation}
where $Y^0_k$, $k=1,\dots,4$ are the coefficients described in Section \ref{tbc}.

Let $(u,v)=\int_a^b uv\, \text{d}x$ be the usual $L_2$ inner product. $V_N^*$ is defined as the \emph{dual} space of $V_N$ so that for every $u\in V_N$ it holds $(u_{xxx},v) = -(u,v_{xxx})$ for every $v\in V_N^*$.

\begin{lemma}
The \emph{dual} space $V_N^*$ of $V_N$ is given by 
\[
\begin{split}
V_N^* = \{v\in\mathcal{P}_N\, |\,  v(b) - Y_4^0v_x(b) + Y_3^0v_{xx}(b) & = 0,\\
v(a) + Y_2^0v_{xx}(a) & = 0,\\
v_x(a) - Y_1^0v_{xx}(a) & = 0\}.
\end{split}
\]
\end{lemma}
\begin{proof}
Integrating the quantity $(u_{xxx},v)$ by parts three times   we get
\begin{equation}
\label{eql1}
\begin{split}
(u_{xxx},v) & = u_{xx}\cdot v\Big|_{x=a}^b - (u_{xx}, v_x) = u_{xx}\cdot v\Big|_{x=a}^b -  u_x\cdot  v_x\Big|_{x=a}^b + (u_x,v_{xx}) \\
& = u_{xx}\cdot v\Big|_{x=a}^b -  u_x\cdot  v_x\Big|_{x=a}^b + u\cdot v_{xx}\Big|_{x=a}^b - (u, v_{xxx}).
\end{split}
\end{equation}
The boundary terms in \eqref{eql1} must vanish. For $x=b$ we have
\begin{equation}
\label{eql2}
u_{xx}(b)\cdot v(b) - u_x(b)\cdot v_x(b) + u(b)\cdot v_{xx}(b) =  u_{xx}(b)\cdot \big(v(b) - Y_4^0 v_x(b) + Y_3^0 v_{xx}(b)\big).
\end{equation}
The last equality is obtained substituting in place of $u_x(b)$ and $u(b)$ the relations with $u_{xx}(b)$ given by the space $V_N$.
Similarly for $x=a$ we have
\begin{equation}
\label{eql3}
\begin{split}
&-u_{xx}(a)\cdot v(a) +  u_x(a)\cdot v_x(a) - u(a)\cdot v_{xx}(a) \\
=& -u_{xx}(a)\cdot v(a) + u_x(a)\cdot v_x(a) - (Y_1^0 u_x(a)+Y_2^0 u_{xx}(a))\cdot v_{xx}(a)  \\
=& -u_{xx}(a)\big(v(a)+Y_2^0 v_{xx}(a)\big) + u_x(a)\big(v(a)-Y_1^0 v_{xx}(a)\big).
\end{split}
\end{equation}
From equations \eqref{eql2} and \eqref{eql3} we get the boundary relations for the dual space $V_N^*$.
\end{proof}

Let $w^m\in V_N$ be the solution of the homogeneous problem associated to  \eqref{eqibvp}, i.e.~where we set $h_i^m = 0$ for $i=1,2,3$. The solution $u^m$ of the inhomogeneous problem is then given by $u^m = w^m + p_2^m$, where $p_2^m$ is the unique polynomial of degree $2$ such that
\[
\begin{split}
& p_2^{m}(-1) - Y_1^0\,p^{m}_{2,x}(-1) - Y_2^0p^{m}_{2,xx}(-1) = h_1^m,\\
& p_2^{m}(1) - Y_3^0\,p_{2,xx}^{m}(1) = h_2^m,\\
& p_{2,x}^{m}(1) - Y_4^0\,p_{2,xx}^{m}(1) = h_3^m.\\
\end{split}
\]

We remark that the polynomial $p_2^m$ depends on $h_i^m$, $i=1,2,3$. Therefore, it must be computed at each time step. For more details we refer the reader to \cite{shen04,zheng08}.

The variational formulation of \eqref{eq4} reads: find $u^m = w^m + p_2^m \in P_N$ with $w^m\in V_N$ such that for every $v\in V_N^*$ it holds 
\[
(u^m,v) + \tau(u^m_{xxx},v) = (u^{m-1}-\tau g(x) u_x^{m-1}, v).
\]

In order to numerically evaluate the $L_2$ inner products $(\cdot,\cdot)$, we have to choose some interior collocation points. A typical choice are the Gauss--Lobatto points. These points are efficient for solving second order differential equations. However, due to the lack of symmetry of the considered third order problem, a better option is to choose as interior collocation points $\{x_k\}_{k=2}^{N-1}$ the roots of the Jacobi polynomial $P^{(2,1)}_{N-2}(x)$,  $x_1 = a$, $x_N=b$, see \cite{abramowitzstegun64,ma99}. 

The discrete inner product associated to the Gauss--Jacobi quadrature rule is given by
\begin{multline}
\label{eq21}
(u,v)_N = \sum_{k=2}^{N-1} \omega_k u(x_k)v(x_k) + \\ \omega_1 u(x_1)v(x_1)  + \omega_N u(x_N)v(x_N)  +\omega'_N \partial_x (uv)(x_N),   
\end{multline}
where the associated weights are 
\[
\begin{split}
& \omega_k = \frac{4}{N^2-1}  \bigg(\frac{2N+1}{N+2}\bigg)^2 \frac{1}{(1-x_k)\,[P^{(2,1)}_{N-1}(x_k)^2]},\quad k=2,\dots,N-1,\\
& \omega_1 = \frac{2}{N^2-1},\\
& \omega_N = \frac{4}{N^2} + \frac{8}{(N-1)N^2(N+1)} \sum_{k=1}^{N-1}\frac{1}{x_N-x_k},\\
& \omega'_N =  -\frac{8}{(N-1)N^2(N+1)}.
\end{split}
\]
The employed quadrature rule integrates exactly polynomials up to degree $2N-~2$. A detailed description of the generalized quadrature rule for third order problems used in this work can be found in \cite{huang92}.

We substitute $u^m = w^m + p_2^m$ and we use the fact that $w^m_{xxx} = u^m_{xxx}$, so we obtain 
\begin{equation}
\label{eq3}
(w^m,v)_N + \tau(w^m_{xxx},v)_N = (u^{m-1}-\tau g(x) u_x^{m-1}-p^m_2, v)_N.
\end{equation}
For any $0\leq k\leq N-3$, set
\[
\begin{split}
& \phi_k = L_k + \alpha_k L_{k+1} + \beta_k L_{k+2} + \gamma_k L_{k+3}, \\
& \psi_k = L_k - \alpha_k L_{k+1} + \beta_k L_{k+2} - \gamma_k L_{k+3}, 
\end{split}
\]
where $L_k$ is the $k$th Legendre polynomial and $\alpha_k, \beta_k, \gamma_k,$ are uniquely determined so that $\phi_k$ and $\psi_k$ belong to $V_N$ and $V_N^*$, respectively. The sequences $\{\phi_k\}_{k=0}^{N-3}$ and $\{\psi_k\}_{k=0}^{N-3}$ constitute a basis of $V_N$ and $V_N^*$, respectively. 
\begin{lemma}
The basis functions $\phi_k$ and $\psi_k$ satisfy
\[
\begin{split}
& (\partial_x^3 \phi_i, \psi_j) = -(\phi_i,\partial_x^3 \psi_j) = 2(2j+3)(2j+5)\gamma_j\delta_{ij} ,\\
& (\phi_i,\psi_j) = 0,\quad |i-j|>3,
\end{split}
\]
where $\delta_{ij}$ is the Kronecker delta and $(\cdot,\cdot)$ is the usual $L_2$ inner product. 
\end{lemma}
\begin{proof}
By definition, we have 
\[
\partial_x^3\phi_i(x) = \partial_x^3 (L_i(x) + \alpha_i L_{i+1}(x) + \beta_i L_{i+2}(x) + \gamma_i L_{i+3}).
\]
The Legendre polynomials $L_i$ satisfy
\begin{equation}
\label{rel1}
\begin{split}
(L_i,L_j) &= \frac{1}{2i+1}\delta_{ij},\\
\partial_x L_i(x) &= \sum_{\substack{k=0\\k+i \text{ odd}}}^{i-1} (2k+1)L_k(x),\\
\partial_x^2 L_i(x) &= \sum_{\substack{k=0\\k+i \text{ even}}}^{i-2} \bigg(k+\frac{1}{2}\bigg)(i(i+1)-k(k+1))L_k(x),\\
\partial_x^3 L_i(x) & = \sum_{\substack{k=0\\k+i \text{ even}}}^{i-2} \bigg(k+\frac{1}{2}\bigg)(i(i+1)-k(k+1))\partial_x L_k(x)\\
& =\sum_{\substack{k=0\\k+i \text{ even}}}^{i-2} \bigg(k+\frac{1}{2}\bigg)(i(i+1)-k(k+1))\sum_{\substack{j=0\\j+k \text{ odd}}}^{k-1} (2j+1)L_j(x).
\end{split}
\end{equation}
For $j>i$, we have $(\partial_x^3 \phi_i,\psi_j)=0$ since $\partial_x^3 \phi_i$ is a linear combination of Legendre polynomials whose highest degree Legendre polynomial is $L_i(x)$. The conclusion follows by the orthogonality property of the Legendre polynomials. On the other hand,  
for $i<j$, we have 
\[
(\partial_x^3 \phi_i,\psi_j)=-(\phi_i,\partial_x^3\psi_j)=0.
\]

The only case left is $i=j$, where we have
\[
\begin{split}
(\partial_x^3 \phi_i,\psi_i) &= (\partial_x^3 (L_i + \alpha_i L_{i+1} + \beta_i L_{i+2} + \gamma_i L_{i+3}),L_i - \alpha_i L_{i+1} + \beta_i L_{i+2} - \gamma_i L_{i+3})\\
& = \gamma_i(\partial_x^3 L_{i+3},L_i).
\end{split}
\] 
By using the relations \eqref{rel1}, we obtain
\[
\gamma_i(\partial_x^3 L_{i+3},L_i) = \bigg(\sum_{\substack{k=0\\k+i \text{ even}}}^{i+1} \bigg(k+\frac{1}{2}\bigg)((i+3)(i+4)-k(k+1))\sum_{\substack{j=0\\j+k \text{ odd}}}^{k-1} (2j+1)L_j,L_i\bigg).
\]
The only non-zero term is given by $k=i+1,\, j = k-1 = i$, so we have
\[
\gamma_i(\partial_x^3 L_{i+3},L_i) = \gamma_i(2i+3)(2i+5).
\]
By using the orthogonality property of the Legendre polynomials, one can easily prove 
\[
(\phi_i,\psi_j)=0 \text{ for } |i-j|>3,
\]
which is the desired result.
\end{proof}
We remark the fact that the $L_2$ inner product $(p,q)$ coincides with the discrete inner product $(p,q)_N$ for all polynomials $p$, $q$ such that $\deg p+\deg q\leq2N-2$. This is a property of the quadrature rule.

\noindent Since $w^m\in V_N$ we can express it in terms of $V_N$ basis functions 
\[
w^m(x) = \sum_{j=0}^{N-3} \hat{w}^m_j \phi_j(x).
\]
Therefore, the variational formulation \eqref{eq3} with $v=\psi_k$ becomes
\[
\sum_{j=0}^{N-3}\hat{w}^m_j\big[(\phi_j, \psi_k)_N + (\partial_x^3 \phi_j,\psi_k)_N\big] = (u^{m-1}-\tau g(x) u_x^{m-1}-p_2^m,\psi_k)_N 
\]
for $j,k=0,\dots,N-3$.

\noindent Let us define the mass matrix 
\[
M=(m_{j,k}) = (\phi_j,\psi_k)_N
\] and the stiffness matrix 
\[
S=(s_{j,k})=(\partial_x^3 \phi_j,\psi_k)_N
\]
for $j,k=0,\dots,N-3$. 

The discrete inner product  $(\phi_j,\psi_k)_N$ coincides with  the $L_2$ inner product $(\phi_j,\psi_k)$ for $j+k+6\leq 2N-2$ since $\phi_k$ and $\psi_k$ are polynomials of degree $k+3$. We can compute the entries $m_{j,k}$ using the $L_2$ inner product except for 
\[m_{N-4,N-3},\,\,m_{N-3,N-4},\,\,m_{N-3,N-3},
\]
where the discrete inner product must be used. The entries $s_{j,k}$ can be computed using the $L_2$ scalar product only, since $\partial_x^3 \phi_j$ is a polynomial of degree at most $N-3$ and $(\partial_x^3 \phi_j)\psi_k$ does not exceed the degree $2N-2$. 
The matrices $M$ and $S$ are seven-diagonal and diagonal, respectively. Therefore, the linear system 
\begin{equation}
\label{eq18}
(M+\tau S)\hat{u}^m = (\Psi^T,u^{m-1}-\tau g(x) u_x^{m-1}-p_2^m)_N,
\end{equation}
where 
\[
\Psi = 
\begin{pmatrix}
\psi_0(x_2) & \dots &\psi_{N-3}(x_2)\\
\vdots & & \\
\psi_0(x_{N-1}) &\dots &\psi_{N-3}(x_{N-1})\\
\end{pmatrix},\quad (\Psi^T,v)_N = \begin{pmatrix}
(\psi_0,v)_N\\
\vdots \\
(\psi_{N-3},v)_N
\end{pmatrix}
\]
has a small bandwidth and is solvable in $O(N)$ operations. 

\begin{remark}
For a constant $g$ the variational formulation \eqref{eq3} becomes
\begin{equation}
\label{eq12}
\begin{split}
 (u^{m},\psi_k)_N + \tau(u^{m}_{xxx},\psi_k)_N &= (u^{m-1},\psi_k)_N - \tau g(u^{m-1}_x,\psi_k)_N\\
 (M+\tau S) \hat{u}^{m} &= (M-\tau D)\hat{u}^{m-1}.
\end{split}
\end{equation}
The matrix $D=(d_{j,k})$ given by $(\partial_x \phi_j,\psi_k)_N$ is not a banded matrix, but a lower triangular matrix as a straightforward computation shows. The linear system \eqref{eq12} is, therefore, expensive to solve.
\end{remark} 
\begin{remark}
For arbitrary $g$ the linear system \eqref{eq18} can be rewritten as 
\[
\begin{split}
(M+\tau S)\hat{u}^{m} &= (\Psi^T,u^{m-1})_N -\tau(\Psi^T,g(x) u_x^{m-1})_N -(\Psi^T,p_2^m)_N\\
&= M\hat{u}^{m-1} -\tau(\Psi^T,g(x) u_x^{m-1})_N -T\hat{p}_2^m,
\end{split}
\]
where $T$ is $m$ independent. The second term on the right-hand side is expensive to compute, since it involves $O(N^2)$ operations. This number can be reduced to $O(N\log N)$ if the Chebyshev--Legendre dual-Petrov--Galerkin method is implemented, see \cite{don94}. Namely, we introduce $I^C_N$, the interpolation operator based on the Chebyshev--Gauss--Lobatto points, and we replace $g(x) u_x^{m-1}$ with $I^C_N\big(g(x) u_x^{m-1}\big)$. Now, we can use the fast Chebyshev--Legendre transform to compute the coefficients of the second term in $O(N\log N)$ operations. For more details we refer the reader to \cite{ma99}.
\end{remark}

\section{Implementation and numerical results}
\label{sec4}
In this section, we discuss some implementation aspects. We also present  numerical results that show the properties of the numerical scheme derived in the previous sections. In particular, we illustrate the super-polynomial spatial convergence of the employed dual-Petrov--Galerkin method. We show first-order convergence in time and the effect of the transparent boundary conditions on the numerical scheme. We also perform numerical simulations for different choices of the function $g$ and show that the transparent boundary conditions do not cause any reflections at the boundaries.

\subsection{Stability}
The numerical scheme \eqref{eq4} is derived by discretizing in time the dispersive part using the implicit Euler scheme and the advection part by the explicit Euler scheme. In the simpler case where we assume periodic boundary conditions and $g(x)=g$ constant we can easily show stability via Fourier analysis. 
In particular, we write $u^m$ as its Fourier series and we indicate the coefficients of the series with $(\hat{u}^m_k)_{k\in\Z}$. We obtain
\[
(1-\tau ik^3)\hat{u}^{m+1}_k = (1-\tau igk)\hat{u}^m_k.
\]
The squared amplification factor given by
\[
\bigg|\frac{\hat{u}^{m+1}_k}{\hat{u}^m_k}\bigg|^2 = \frac{1+\tau^2g^2k^2}{1+\tau^2 k^6},\quad k=0,1,2,\dots
\]
is always less or equal than $1$ for $|g|\leq1$, and can be bounded by $1+c\tau$ for $|g|>1$ with a constant $c$ depending only on $g$. We thus obtain
\[
|\hat{u}^{m}_k| \leq (1+c\tau)^{m/2} |\hat{u}^0_k|, 
\]
which shows that the scheme is stable. We remark that using explicit Euler for the semi-discrete advection problem would be unstable, but the overall scheme is stable due to the use of the implicit Euler method for the dispersive part. Further, we remark that for $|g|\leq 1$ the problem is unconditionally stable. This simple analysis gives an indication that also for the problem considered in this work, where  $g(x)$  is no longer constant and transparent boundary conditions are employed, the same might hold. The numerical experiments support the above observations.

\subsection{Inverse $\mathcal{Z}$-transform}

A numerical procedure to compute the inverse $\mathcal{Z}$-transform is given in \cite{besse16}. Here we recall the main aspects.
Given a sequence $(u^l)_{l\geq0}$ its $\mathcal{Z}$-transform reads
\begin{equation}
\label{eq14}
\hat{u}(z) = \mathcal{Z}\{(u^l)_{l\geq0}\}(z) = \sum_{l=0}^{\infty} u^l z^{-l},\quad |z|>\rho\geq 1.
\end{equation}
To recover the $l$th element of the sequence starting from $\hat{u}$ one can use Cauchy's integral formula
\begin{equation}
\label{eq17}
u^l = \frac{1}{2i\pi}\oint_{S_r} \hat{u}(z)z^{l-1}\,\text{d}z,\quad r>\rho,
\end{equation}
where $S_r$ is the circle with center $0$ and radius $r$. Approximating the integral by trapezoidal rule with $N$ points, we get
\begin{equation}
\label{eq13}
u^l \approx r^l\mathcal{F}^{-1}\{\mathbf{U}\}(l),\quad 0\leq l<N,
\end{equation}
where $\mathcal{F}^{-1}$ denotes the inverse discrete Fourier transform and $\mathbf{U}=\{U_k\}_{0\leq k\leq N}$ is the $N$-periodic sequence with $U_k = \hat{u}(r e^{2\pi i k/N})$. 

The radius $r$ and the number of points $N$ should be chosen in order to guarantee good approximations of the coefficients $u^l$. 
The choice of the radius $r$ has been studied in \cite{lubich88, zisowsky03}. Notice that $r>1$ must be sufficiently close to $1$ so that the inverse $\mathcal{Z}$-transform does not become numerically unstable when $l$ grows. A possibility to suppress this instability problem is to compute the inverse $\mathcal{Z}$-transform with a quadruple precision algorithm as proposed in \cite{besse16}. A more effective remedy to overcome the instability has been proposed in \citep{besse16a}. There the authors exploit the relation between the coefficients of the polynomial \eqref{cpol} and its roots. For simplicity, we only consider the former approach in the present work. We set $r=\mathrm{e}^{C\tau}$ with appropriately chosen $C>0$ and $N= m\lceil |\log(10^{-7})|\rceil$, where $m$ is the number of time steps. A complete analysis of the optimal choice of the radius is not the goal of this work. For a thorough discussion we refer the reader to \cite{lubich88,zisowsky03}. 

The number of points $N$ discretizing \eqref{eq17} should be large enough in order to achieve good approximations of the coefficients $u^l$ in \eqref{eq13}. It is known that the discrete Fourier transform using $N$ points is a good approximation of the Fourier transform for all the $u^l$ such that $|l|\leq N/2$. Therefore, $N\geq 2l$ is a good choice. 

In formulae \eqref{eq15} convolution products appear for computing the transparent boundary conditions. The computation of such quantities becomes expensive as $l$ grows. Therefore, for sufficiently long time simulations, the computational cost will be dominated by the number of time steps. In order to keep the computational cost of the convolution sums to a minimum, one can approximate them, for example with a sum of exponentials approach. For a thorough discussion about the discrete transparent boundary conditions and the sum of exponential approach we refer the reader to \cite{arnold03,besse16}.  

\subsection{Numerical results}
For the numerical tests we consider the problem
\begin{equation}
\label{eq16}
\begin{cases}
u_t + g(x) u_x + u_{xxx} = 0,\quad (t,x)\in[0,T]\times[-6,6],\\
u(0,x) = \mathrm{e}^{- x^2},\\
\end{cases}
\end{equation}
subject to transparent boundary conditions. The spatial domain is chosen sufficiently large so that the initial data $u(0,x)$ is close to $0$ for $x=\pm6$, in particular $|u(0,\pm6)|<10^{-15}$. Therefore, even if the initial data has not compact support, the effects produced by $u(0,x)$ at the boundaries are negligible. 

As first example we consider problem \eqref{eq16} with 
\[
g(x) = 0,\quad T=2.
\]
The same problem is considered in \cite{besse16, zheng08}. In this case the exact solution is available and given by 
\[
u_{\text{exact}}(t,x) = E(t,x)*u(0,x),\quad E(t,x) = \frac{1}{\sqrt[3]{3t}}\mathrm{Ai}\Big(\frac{x}{\sqrt[3]{3t}}\Big),
\]
where $*$ is the convolution on the entire real line and $\mathrm{Ai}$ denotes the Airy function. Notice that the spatial domain is $[-6,6]$ whereas the pseudo-spectral approach derived in Section \ref{secPseudo} works only for $x\in[-1,1]$. Therefore, we first scale problem \eqref{eq16} to $[-1,1]$. 
Then, we compute the numerical solution to the scaled problem and finally we scale back to the original domain. In Fig.~\ref{fig1} we show the result of the numerical simulation.   No reflections can be seen at the boundaries. 
\begin{figure}
\centering
\includegraphics[scale=.39]{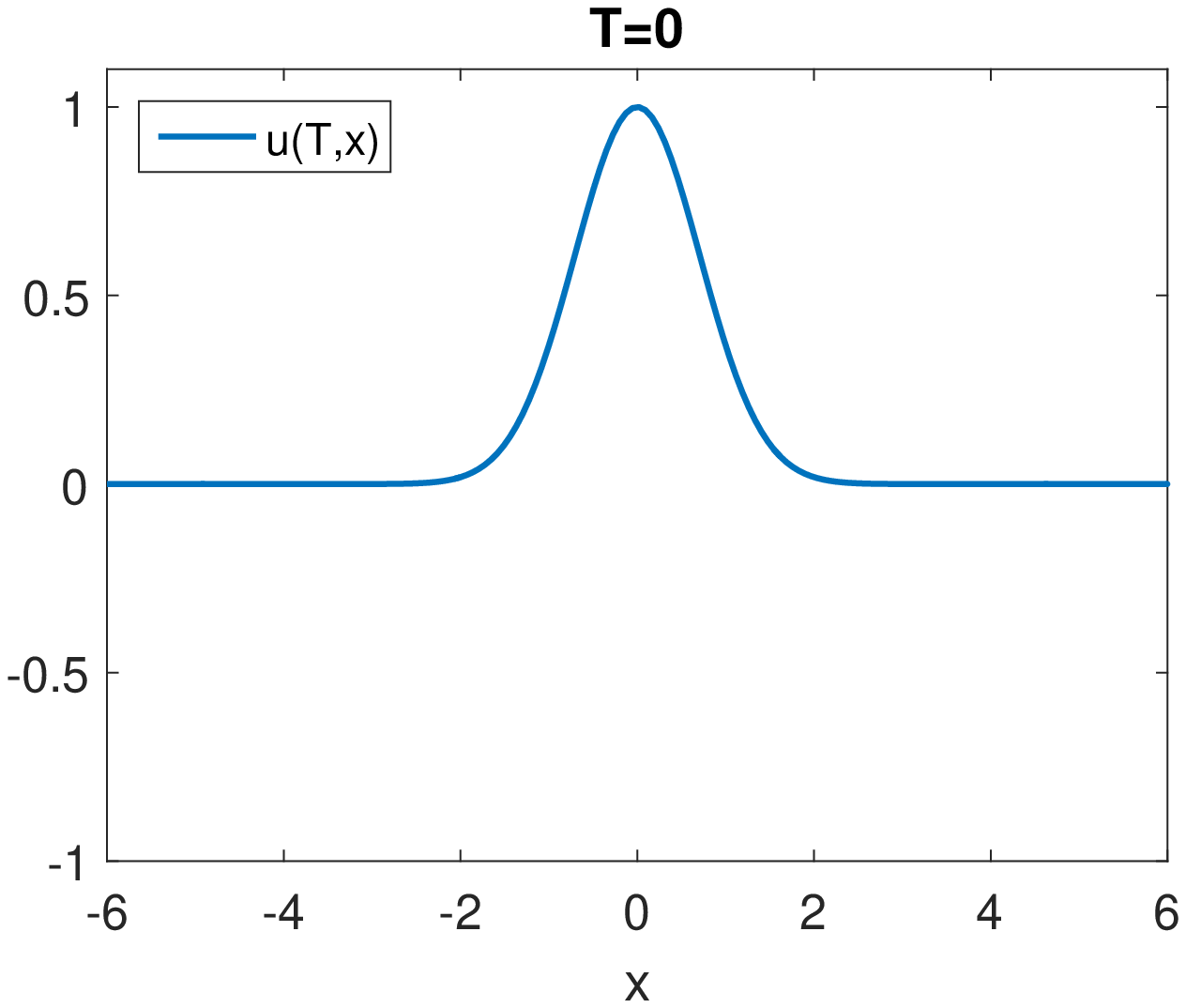}
\includegraphics[scale=.39]{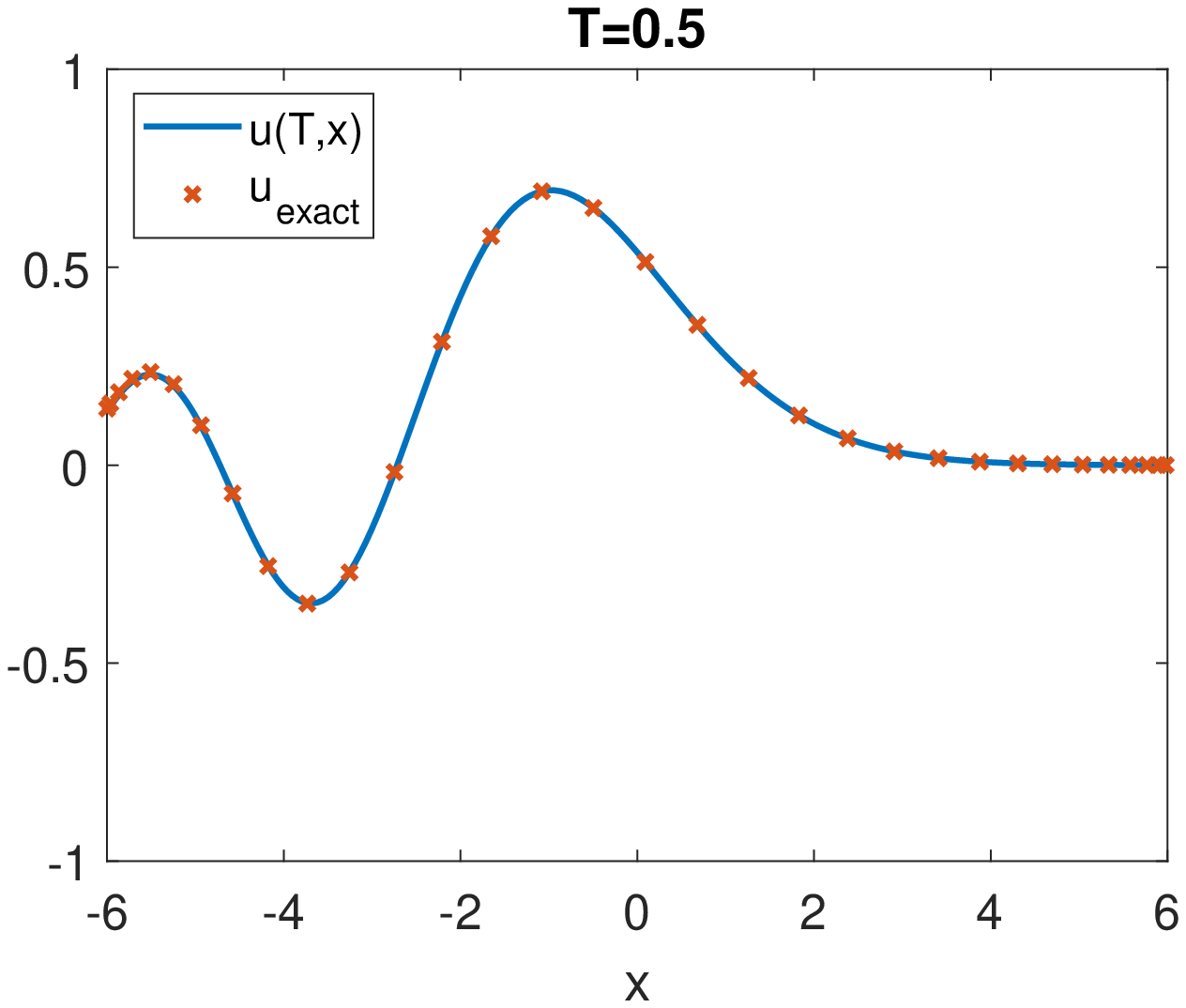}
\includegraphics[scale=.39]{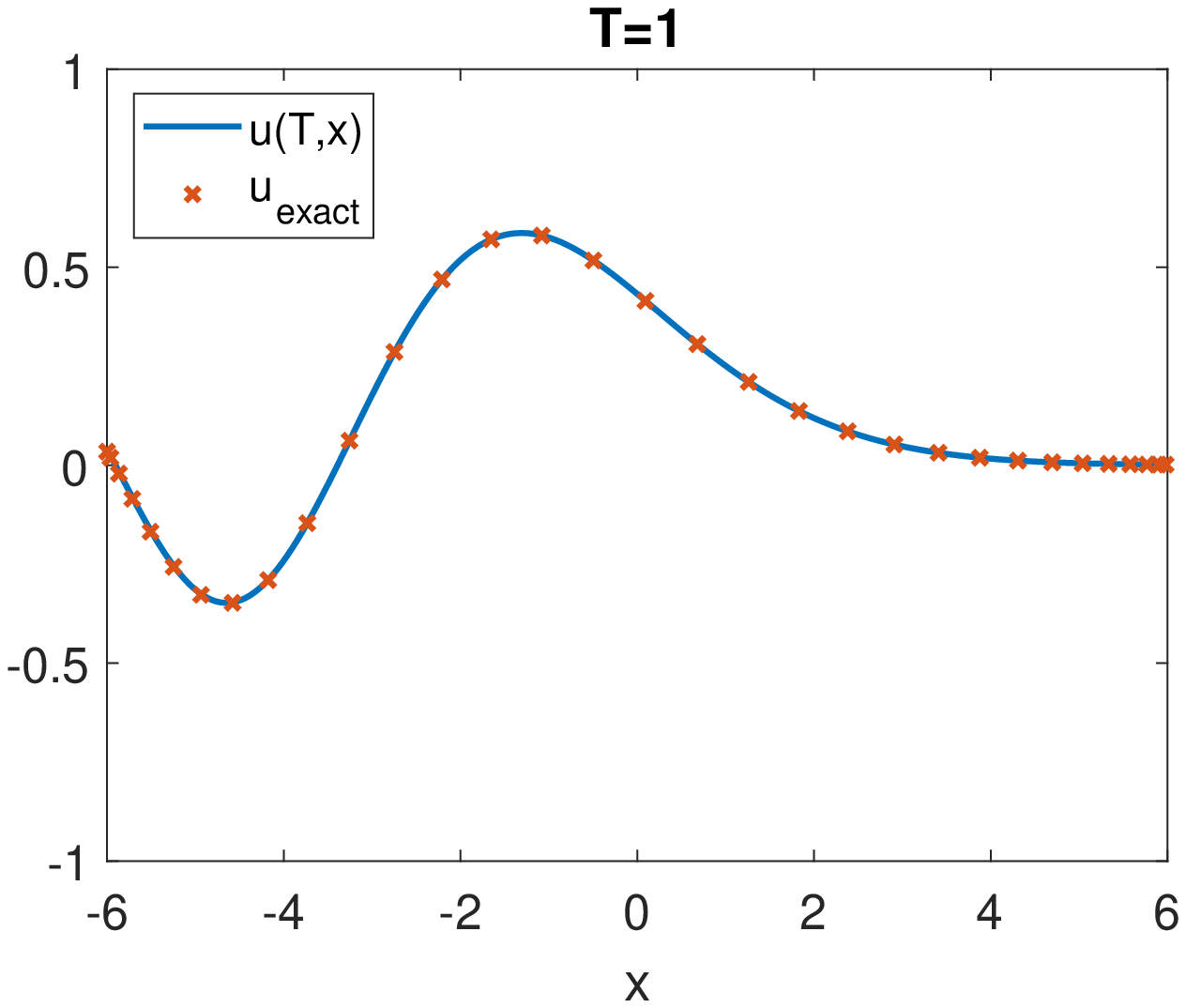}
\includegraphics[scale=.39]{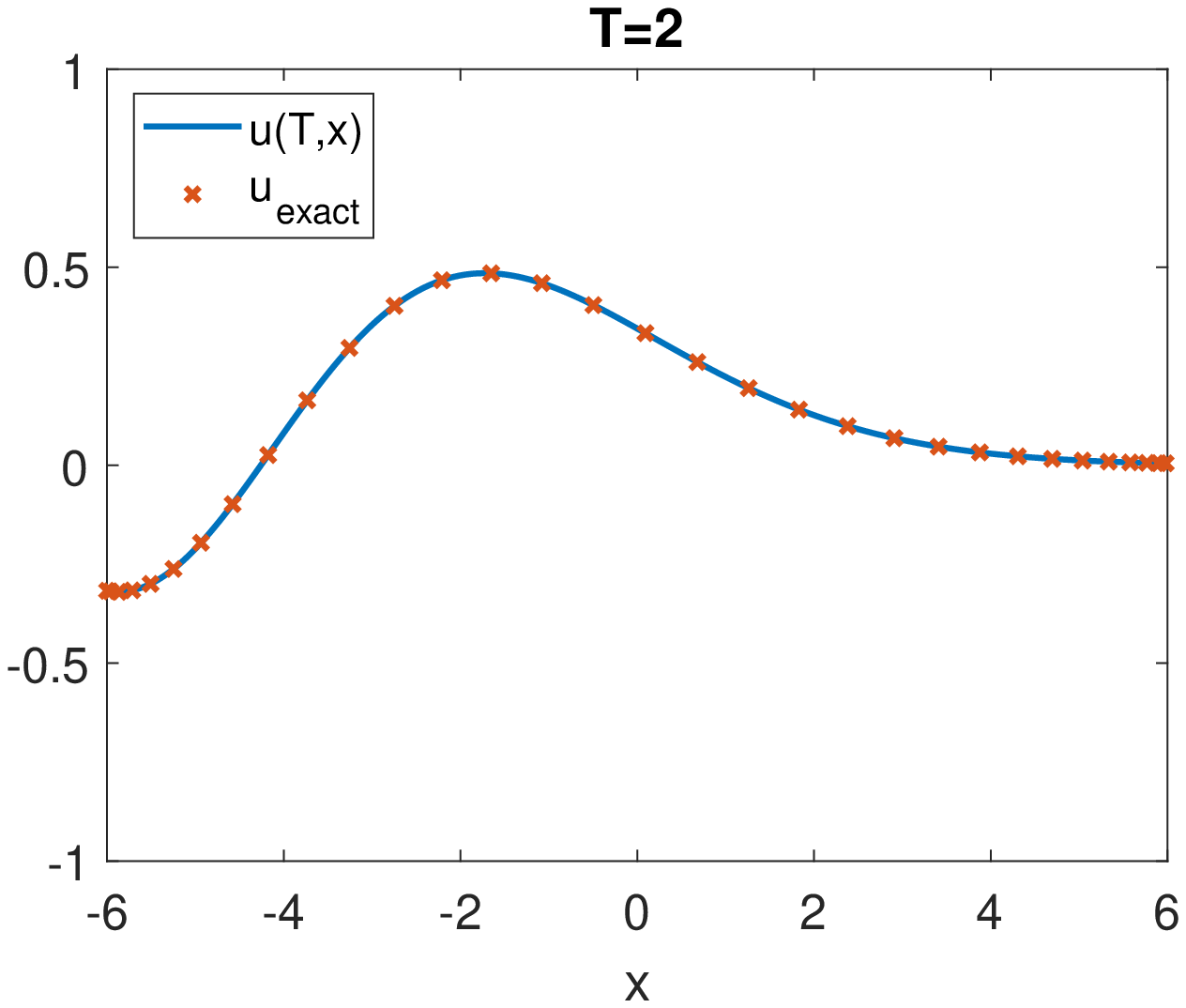}
\caption{Exact and numerical solution of problem \eqref{eq16} for $g(x)=0$ and final time $T=2$. We use $m=2^{11}$ time steps and $N=2^8$ points for the space discretization.}
\label{fig1}
\end{figure}

In Fig.~\ref{fig2} we present the numerical solution for 
\[
g(x) = 6,\quad T=2.
\]
This example is also considered in \cite{besse16}, where a numerical solution was provided by employing a finite difference spatial discretization. Our approach achieves very accurate spatial result using a modest number of grid points. Therefore, it gives an improvement compared  to the finite difference approach, where a fairly large amount of spatial gird points are used. We notice the effect of the advection term that shifts the solution to the right.
The exact solution can be obtained via Fourier transform, see \cite{besse16}. Again no reflections can be observed at the boundaries. 
\begin{figure}
\centering
\includegraphics[scale=.39]{u0}
\includegraphics[scale=.39]{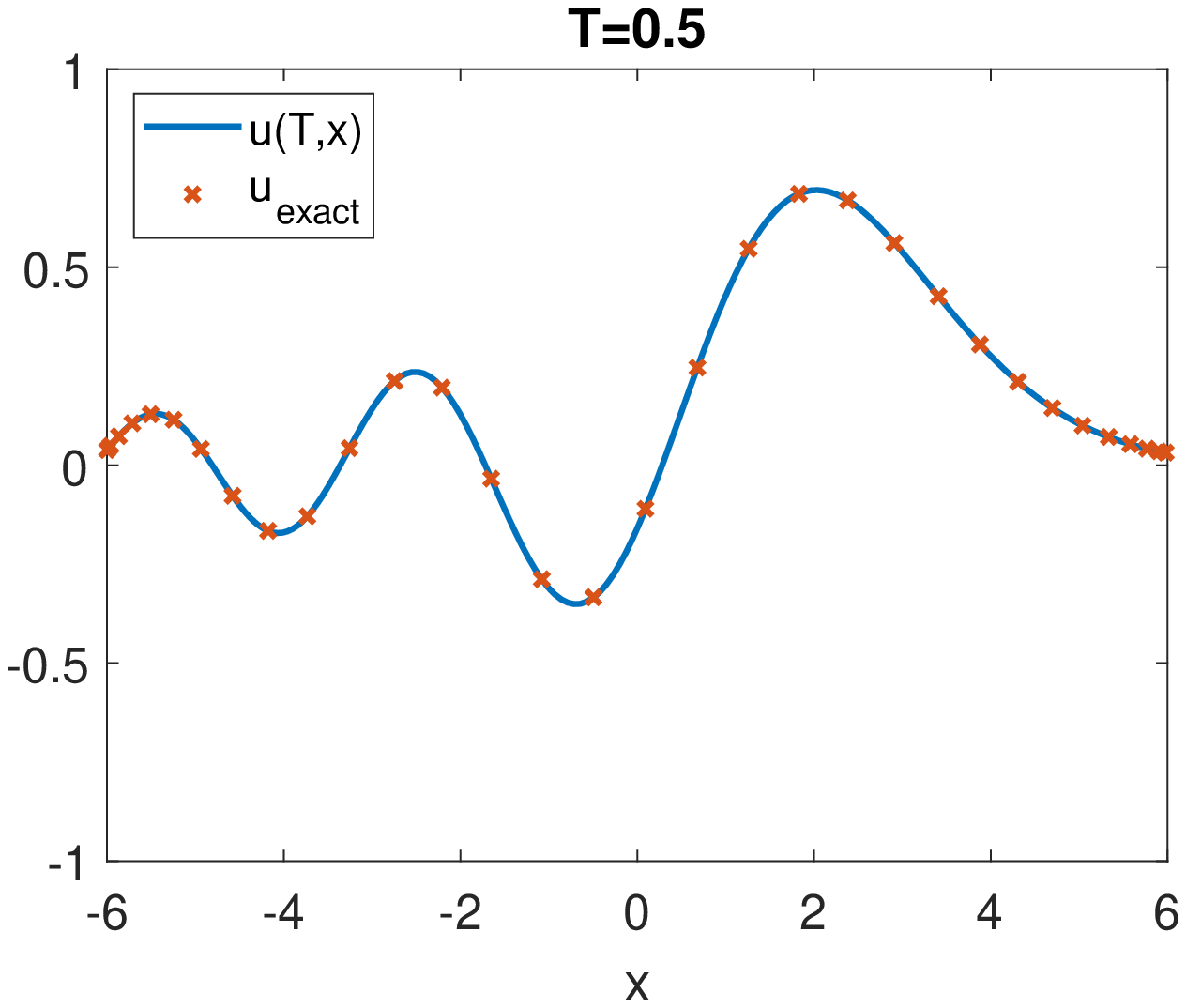}
\includegraphics[scale=.39]{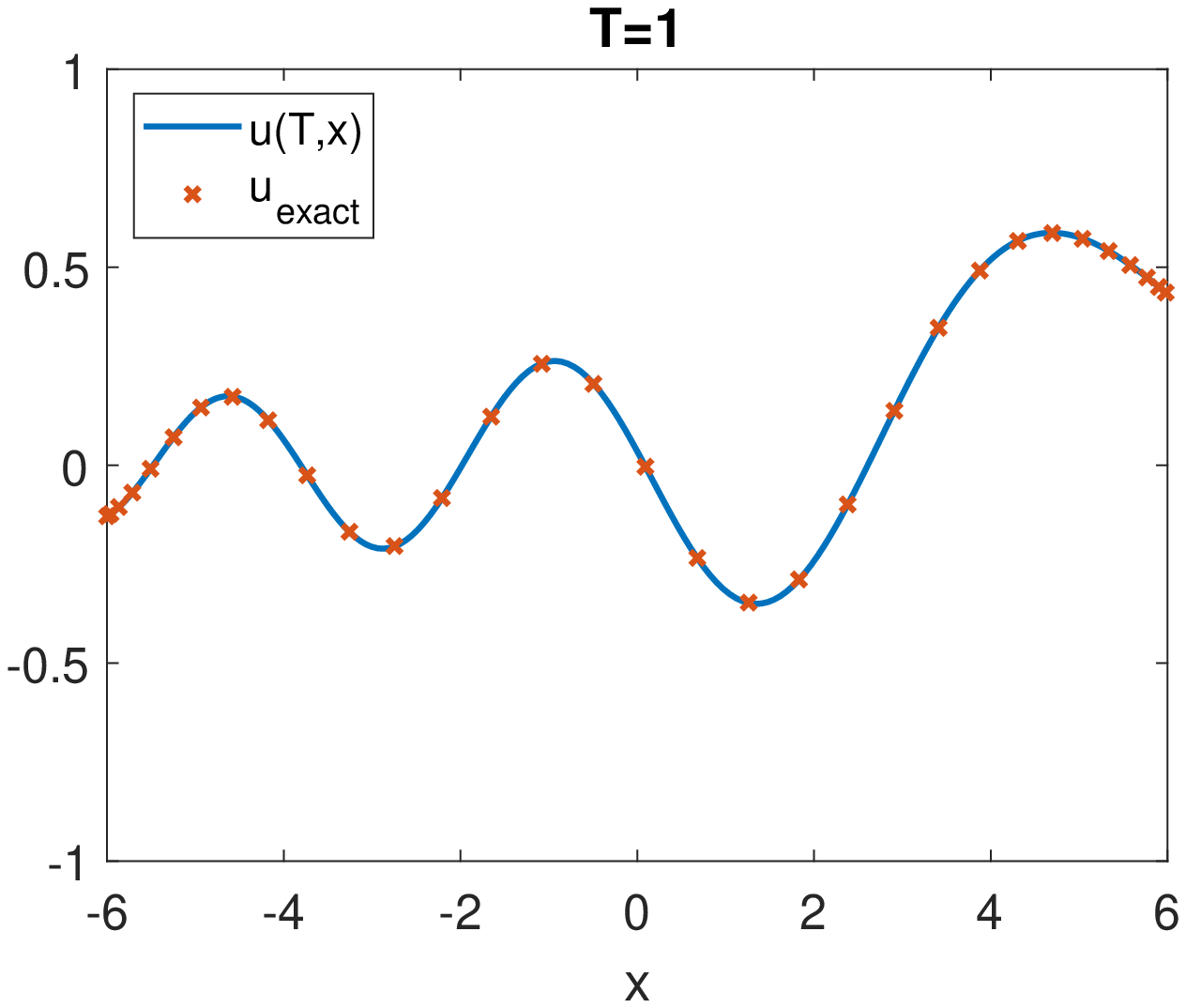}
\includegraphics[scale=.39]{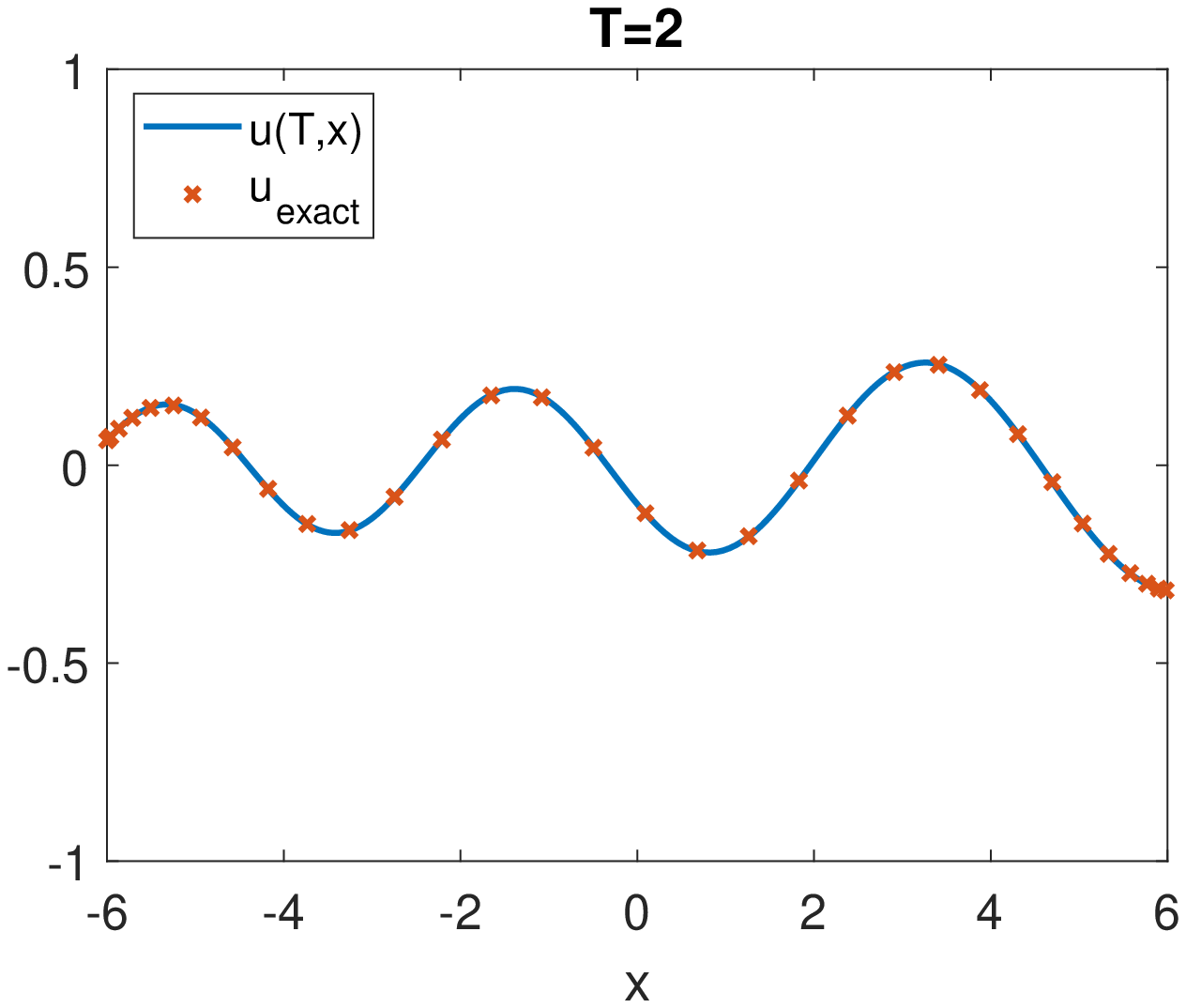}
\caption{Exact and numerical solution of problem \eqref{eq16} for $g(x)=6$ and $T=2$. We use $m=2^{14}$ time steps and $N=2^8$ points for the space discretization.}
\label{fig2}
\end{figure}

In Fig.~\ref{fig5} we present the numerical solution for 
\[
g(x) = -6,\quad T=1.
\]
This example has a negative velocity $g$ and the exact solution can be obtained via Fourier transform. We notice that the advection term shifts the solution to the left, so no dynamic is happening at the right boundary. The solution leaves the domain through the left boundary without reflections.
\begin{figure}
\centering
\includegraphics[scale=.39]{u0}
\includegraphics[scale=.39]{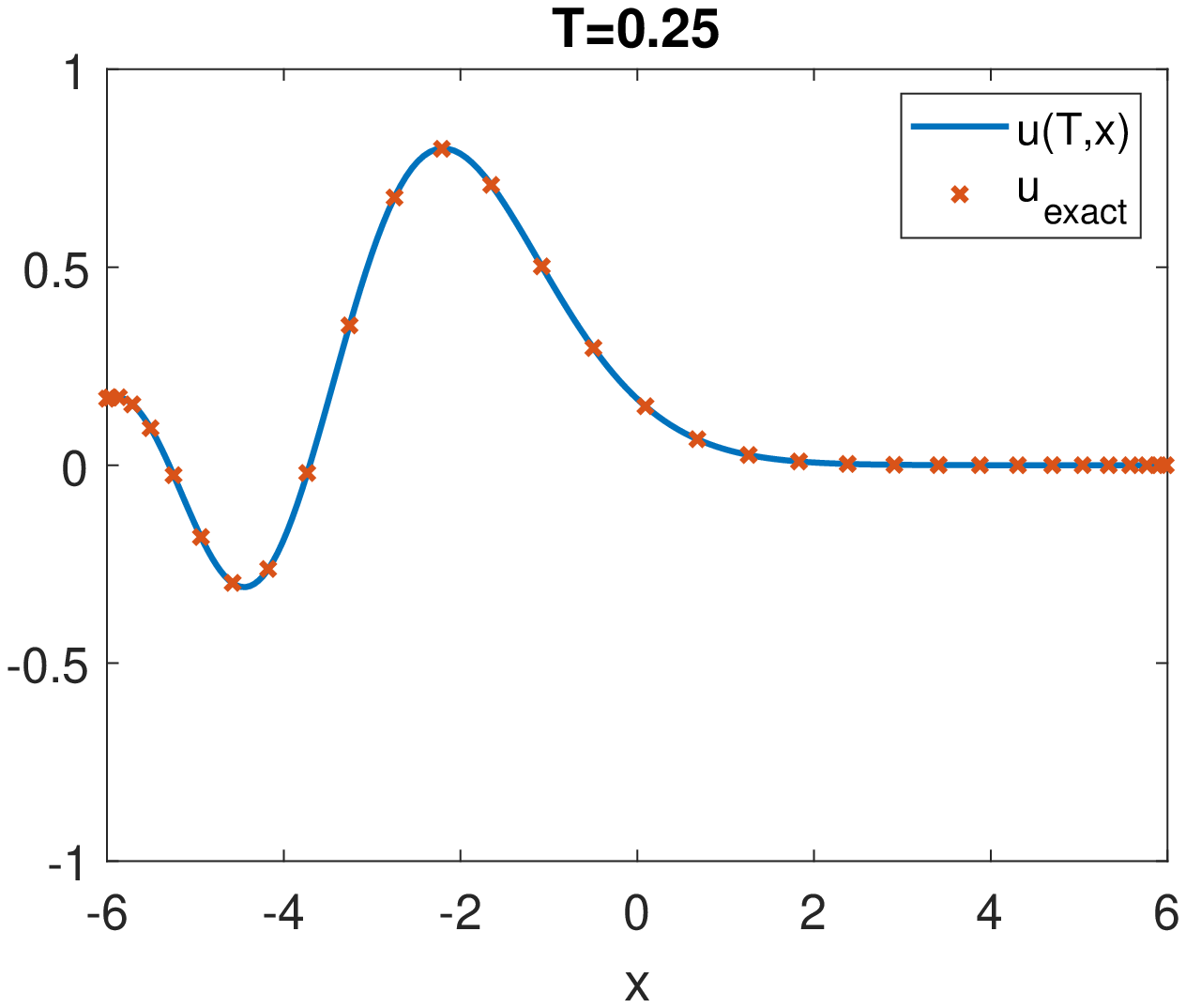}
\includegraphics[scale=.39]{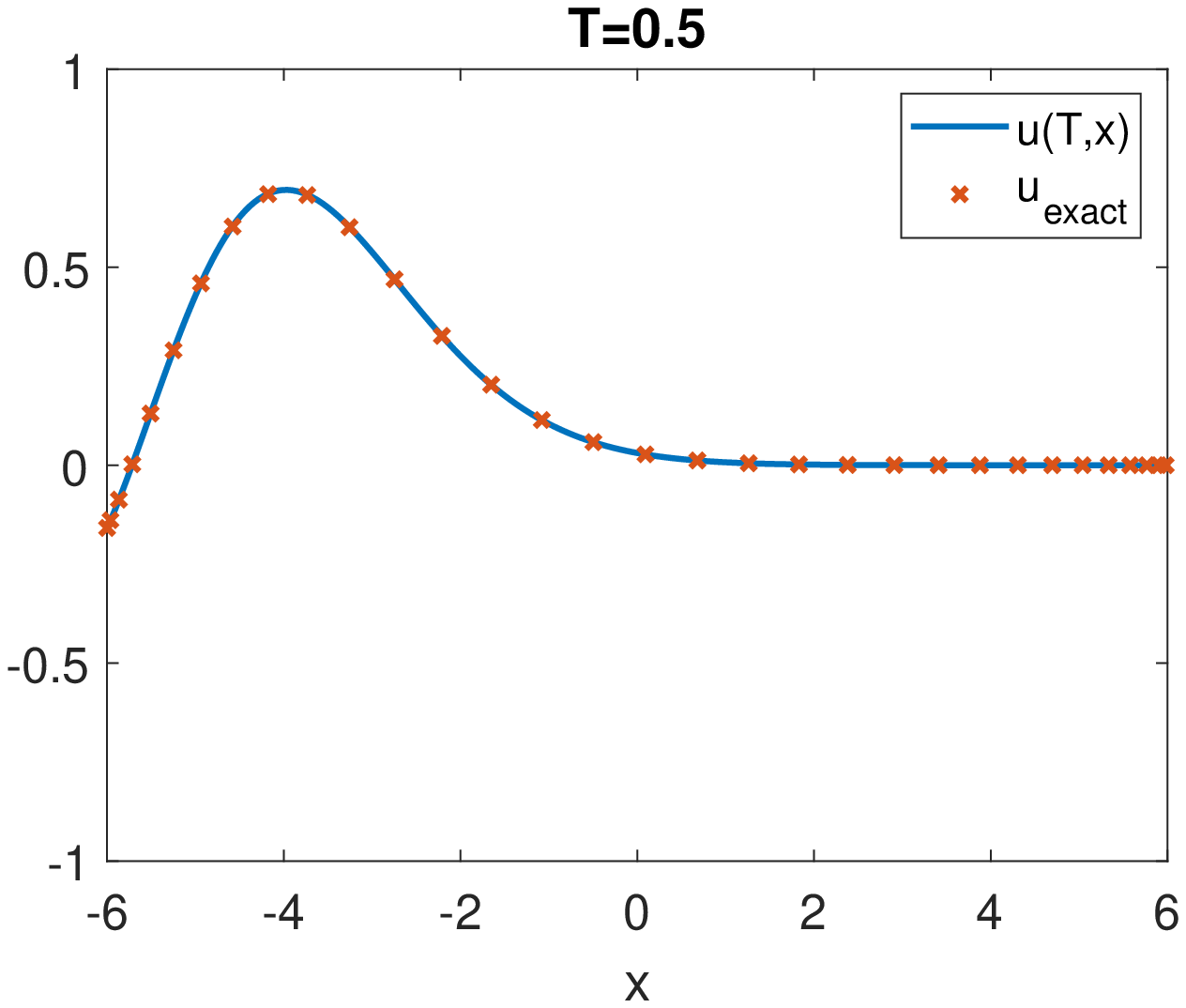}
\includegraphics[scale=.39]{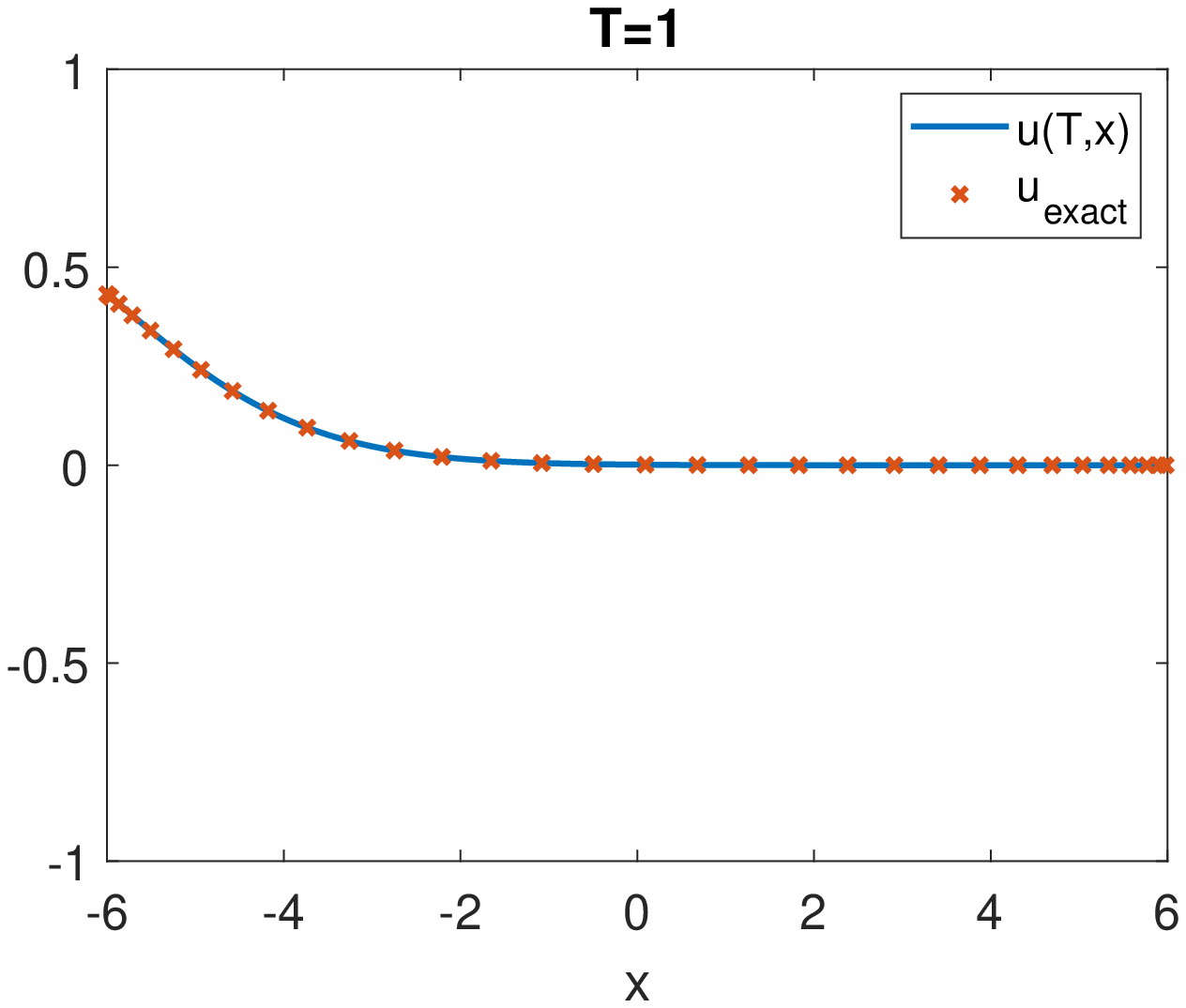}
\caption{Exact and numerical solution of problem \eqref{eq16} for $g(x)=-6$ and $T=~1$. We use $m=2^{12}$ time steps and $N=2^8$ points for the space discretization.}
\label{fig5}
\end{figure}

Finally, we consider a third example where $g$ is no longer constant and given by
\[
g(x)=\pi\bigg(1+\cos\bigg(\frac{\pi(x+6)}{12}\bigg)\bigg),\quad T=1.
\]
The function $g$ is chosen in such a way that its extension to the whole domain $\mathbb{R}$ is smooth. Therefore, since $g$ is assumed to be constant outside the finite computational domain $[-6,6]$, we ask for $g_x(\pm 6) = 0$. In Fig.~\ref{fig3} we plot the function $g$ as well as the numerical solution $u(T,x)$. We notice an advection for $x\in[-6,0]$ and a decay of $g$ for $x\in[0,6]$. One expects that the solution will be shifted to the right due to the advection term for $x\in[-6,0]$ and that there will be almost no effects in the second part of the domain, i.e., for $x$ close to $6$. This behaviour can be clearly seen in the plot, where, at the right boundary, the solution stays close to $0$.     
\begin{figure}
\centering
\includegraphics[scale=.39]{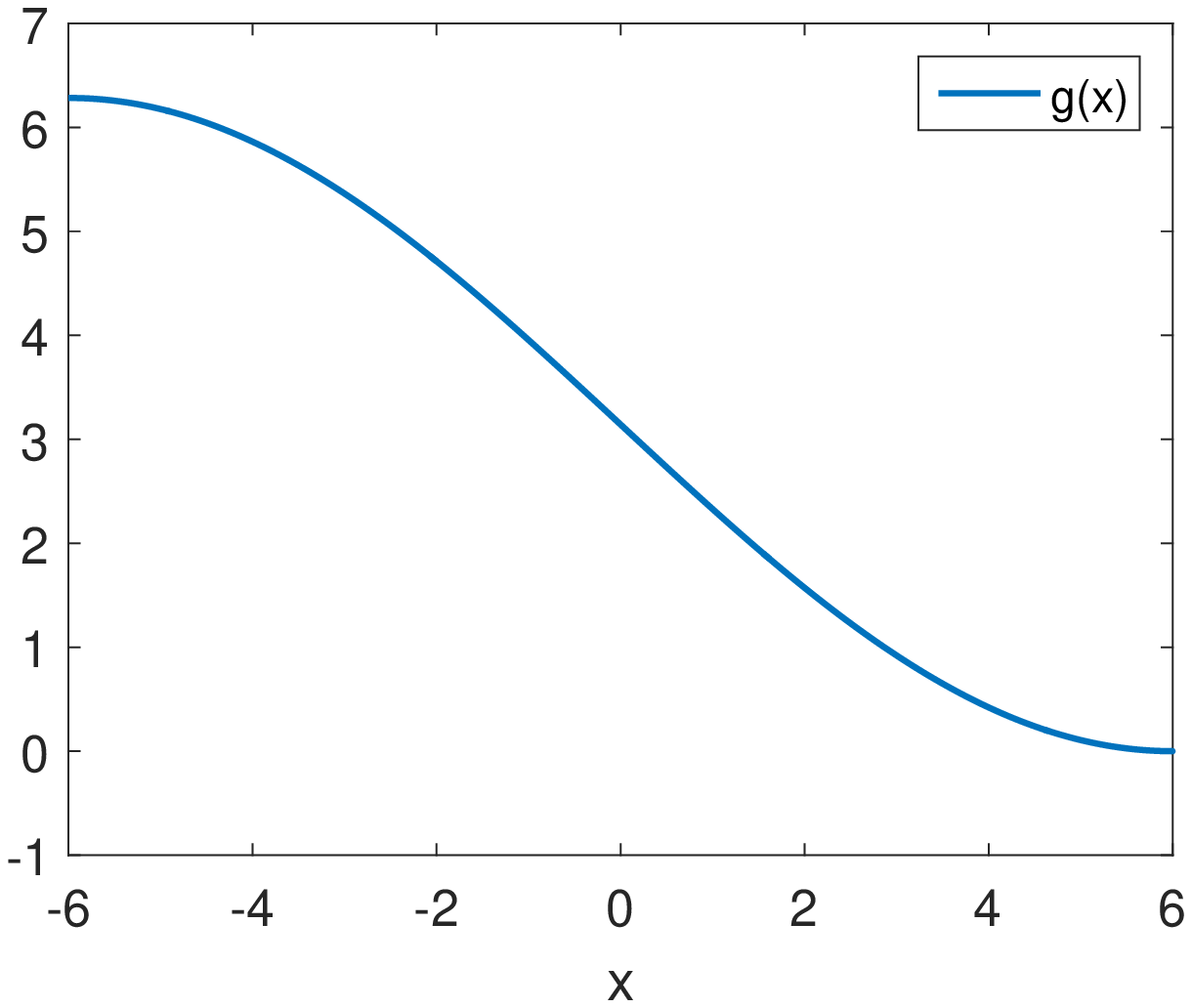}
\includegraphics[scale=.39]{u0}
\includegraphics[scale=.39]{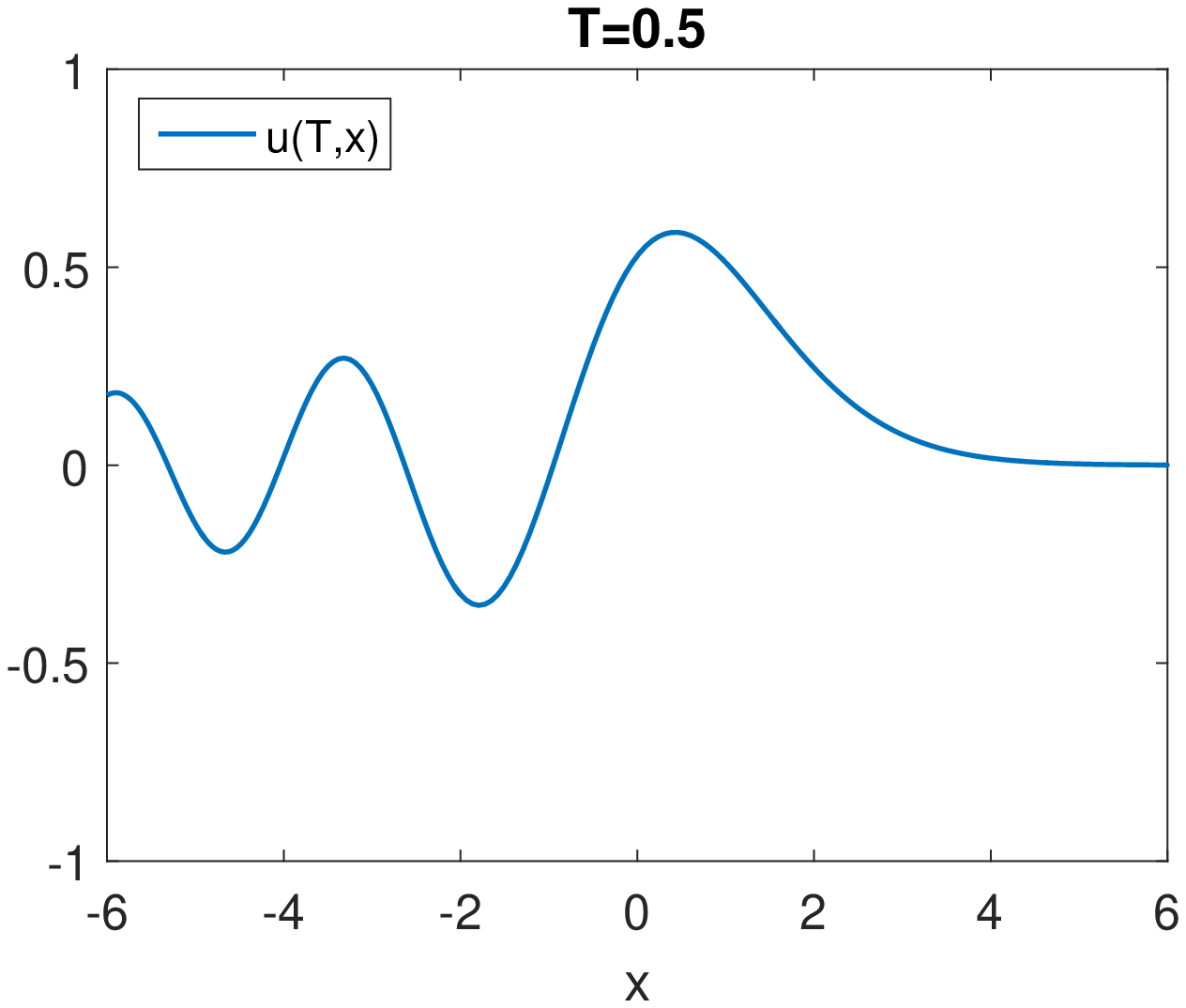}
\includegraphics[scale=.39]{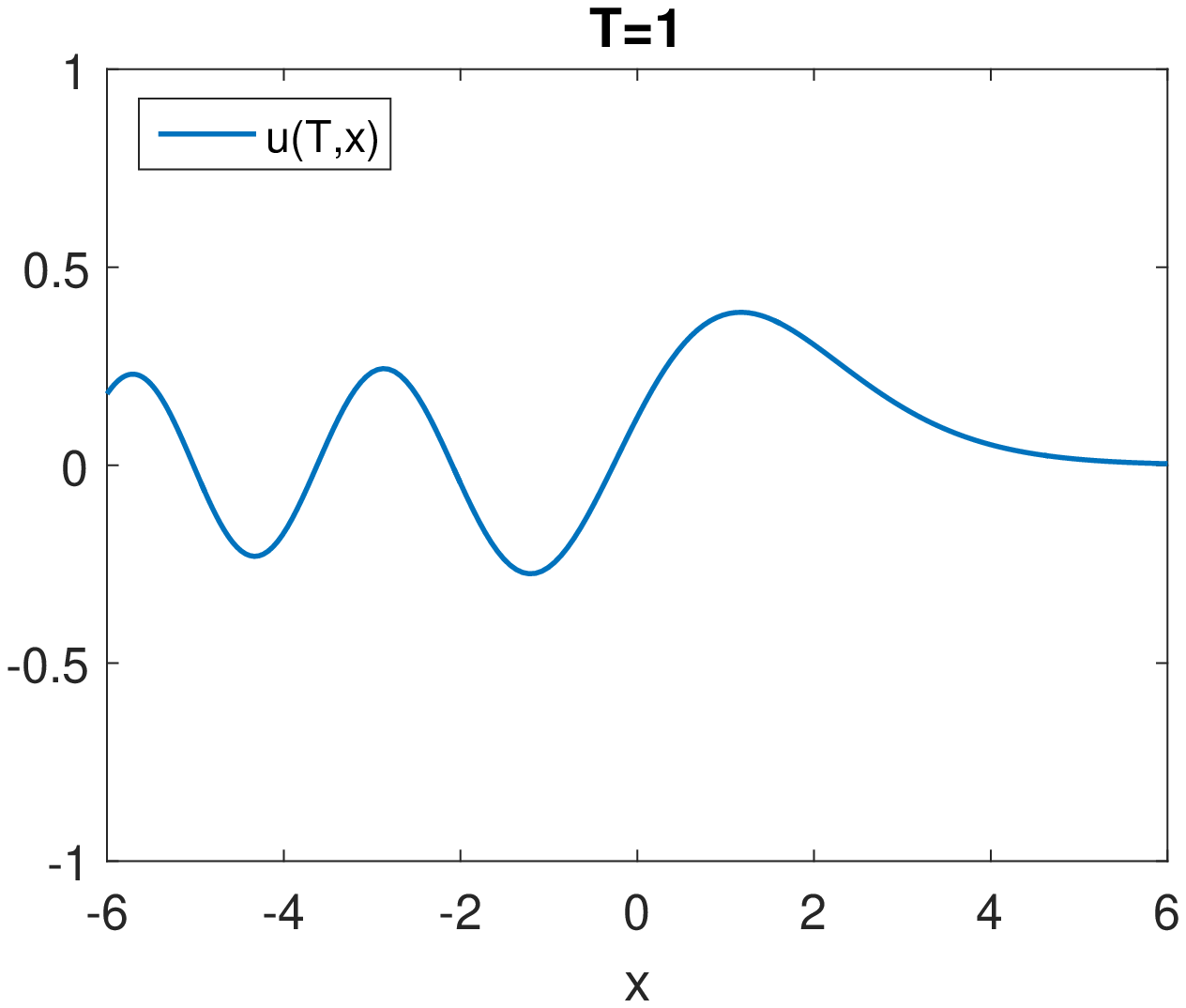}
\caption{In the top-left picture the function $g(x)=\pi\big(1+\cos(\frac{\pi(x+6)}{12})\big)$ is plotted. In the top-right picture we plot the initial data $u_0(x)$. In the second row the numerical solutions for times $T=0.5$ and $T=1$ are plotted. We use $m=2^{13}$ points for the time discretization and $N=2^8$ for the space discretization.}
\label{fig3}
\end{figure}

In Fig.~\ref{fig4} (left picture) we show the behaviour of the spatial error of the numerical solution compared to a reference solution for a fixed time $T=0.5$ and time step $\tau = 2^{-14}$ for the three different choices of $g(x)$. As a reference solution we consider the numerical approximation employing $N=64$ grid points. In this setting the spatial error dominates the time error. Let    
\[
\text{err}^m = \sqrt{\sum_{i=1}^N \bigg(\frac{u^m_{\text{ref}}(x_i)-u^m_{N}(x_i)}{u^m_{\text{ref}}(x_i)}\bigg)^2}
\]
be the relative $\ell^2$ spatial error computed at time $t_m = \tau m$ for $N\in[16,48]$. The grid points $x_i$ are the collocation points as in \eqref{eq21}. From $\text{err}^m$ we compute the $\ell^2$ error in time as 
\[
\Vert\text{err}\Vert_{\ell^2} = \sqrt{\tau \sum_{m=1}^M (\text{err}^m)^2}.
\]
We observe in a semilogy plot the typical supergeometric convergence $\exp(-cN^2)$ for analytic functions by spectral methods, see \cite{boyd13}. The error plot confirms the theoretical derivations of the previous sections. In particular, with less then $50$ spatial points we can achieve numerical solutions with an error less than $10^{-7}$.
\begin{figure}
\centering
\includegraphics[scale=.39]{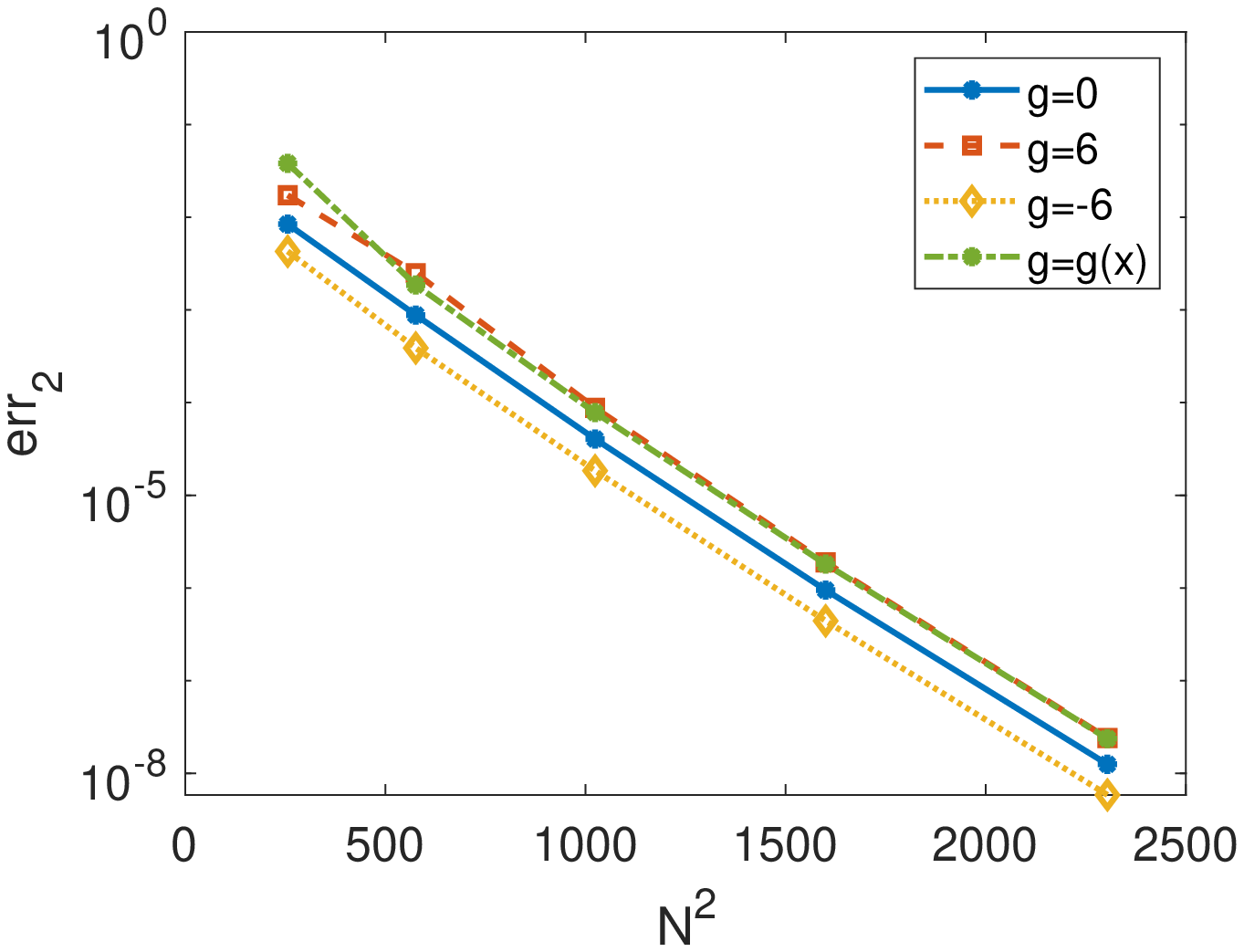}
\includegraphics[scale=.39]{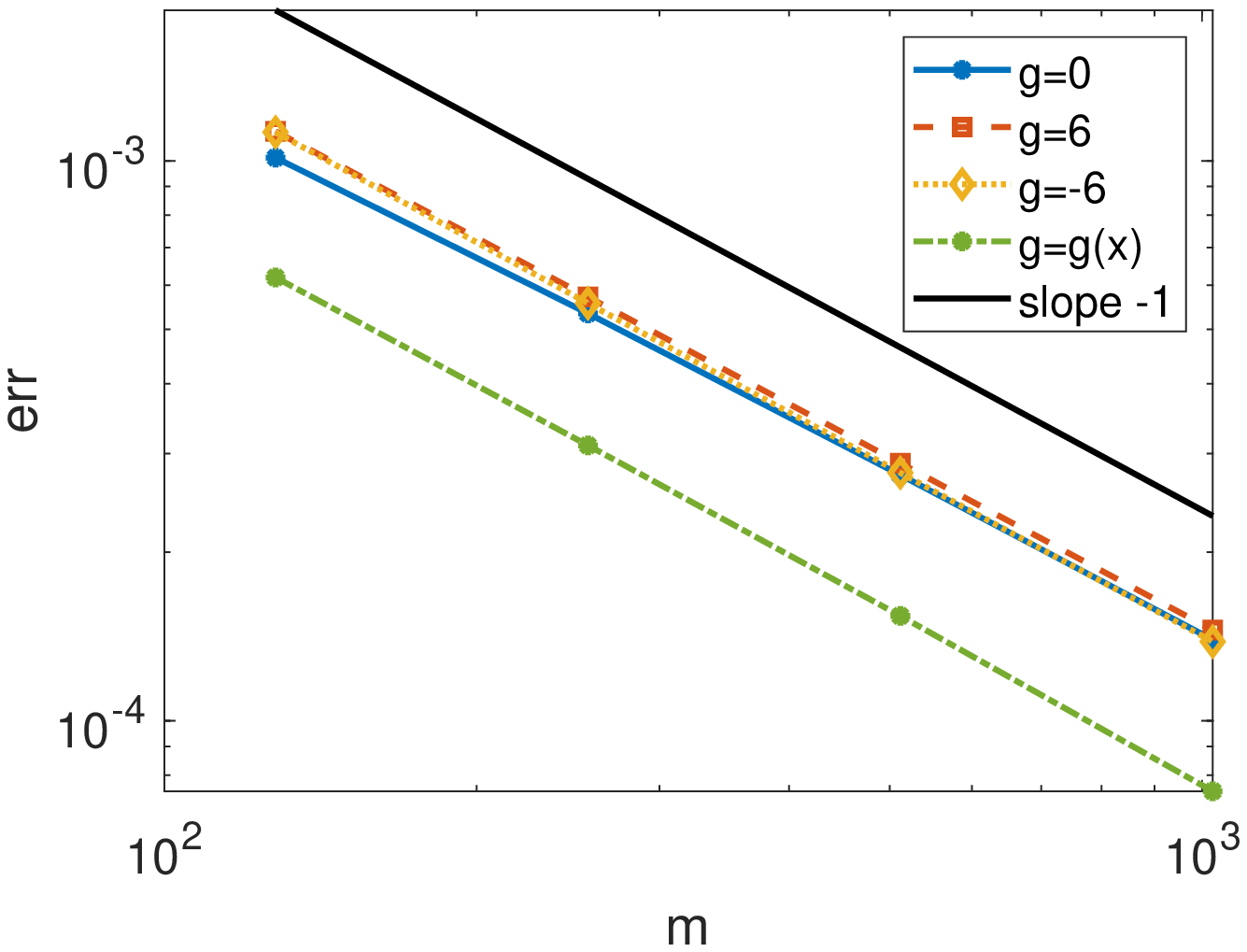}
\caption{In the left picture we plot the relative spatial error for $g=0$, $g=\pm 6$ and $g(x)=\pi\big(1+\cos(\frac{\pi(x+6)}{12})\big)$ for $N\in[16,48]$ at $T=0.5$ with $m =2^{13}$ time steps. Notice that the lines for $g=6$ and $g=g(x)$ are overlapping. In the right we show a double logarithmic plot of the time relative error at $T=0.5$ for different number of time steps $m=2^7,\dots,2^{10}$. The lines for constant $g$ are overlapping.}
\label{fig4}
\end{figure}

In Fig.~\ref{fig4} (right picture) we show the time convergence for $g(x) = 0$, $g(x)=\pm 6$, $g(x)~=~\pi\big(1+\cos(\frac{\pi(x+6)}{12})\big)$ and $N=64$ space grid points. For the case where $g$ is constant, we consider the exact solution. For $g=g(x)$ we compute a reference solution employing $m=2^{14}$ time steps and compare the $\ell^2$ norm evaluated at $T=0.5$ for different $m=2^7,\dots,~2^{10}$. The order is one, as expected. 

\bibliographystyle{siam}

\end{document}